\documentclass[reqno]{amsart}

\usepackage[utf8]{inputenc}
\usepackage[LGR,T1]{fontenc}
\usepackage[greek,english]{babel}

\usepackage{amsfonts}
\usepackage{amsmath}
\usepackage{amssymb}
\usepackage{amsthm}
\usepackage[backend=biber, bibencoding=utf8, giveninits=true,
maxbibnames=99]{biblatex}
\usepackage[shortlabels]{enumitem}
\usepackage[a4paper]{geometry}
\usepackage{graphicx}
\usepackage{mathrsfs} 
\usepackage{mathtools}
\usepackage{multicol}
\usepackage{thmtools}
\usepackage{tikz}
\usepackage{tikz-cd}
\usepackage{xcolor}
\usepackage{xspace}

\DeclareRobustCommand{\alphabeta}[1]{\foreignlanguage{greek}{\alph{#1}}}
\AddEnumerateCounter%
{\alphabeta}%
{\alphabeta}%
{\foreignlanguage{greek}{\textomega}}

\definecolor{tolblue}{RGB}{68,119,170}
\definecolor{tolgreen}{RGB}{34,136,51}
\definecolor{tolred}{RGB}{238,102,119}

\usepackage{thmtools}

\usepackage{zref-clever}
\zcsetup{nameinlink=true, capfirst}
\newcommand{\cref}[1]{\zcref[S]{#1}}

\usepackage[unicode, hypertexnames=false]{hyperref}
\hypersetup{
  colorlinks=true,
  citecolor=black,
  filecolor=black,
  linkcolor=blue,
  urlcolor=red
}

\declaretheorem[
  title=Theorem,
  refname={Theorem, Theorems},
  Refname={Theorem, Theorems},
  within=section
]{theorem}

\declaretheorem[
  title=Corollary,
  refname={Corollary, Corollaries},
  Refname={Corollary, Corollaries},
  sibling=theorem
]{corollary}

\declaretheorem[
  title=Lemma,
  refname={Lemma, Lemmas},
  Refname={Lemma, Lemmas},
  sibling=theorem
]{lemma}

\theoremstyle{definition}

\declaretheorem[
  title=Open Problem,
  refname={Open Problem, Open Problems},
  Refname={Open Problem, Open Problems},
  sibling=theorem
]{openproblem}

\setlist[enumerate, 1]{wide, leftmargin=*, labelindent=0pt}
\usetikzlibrary{decorations.markings}
\DeclareMathOperator{\dom}{dom}
\DeclareMathOperator{\rank}{rank}
\DeclareMathOperator{\im}{im}
\DeclareMathOperator{\id}{id}
\DeclareMathOperator{\Stab}{Stab}
\DeclareMathOperator{\supp}{supp}

\renewcommand{\to}{\longrightarrow}
\newcommand{\J}{\mathscr{J}}
\newcommand{\N}{\mathbb{N}}
\newcommand{\ve}{\zeta}
\newcommand{\vep}{\eta}
\newcommand{\set}[2]{\{#1\mid#2\}}
\newcommand{\presn}[2]{\langle#1\mid#2\rangle}
\newcommand{\defn}[1]{\textit{\textbf{#1}}}
\renewcommand{\P}{\mathcal{P}}
\newcommand{\GAP}{\textsc{GAP}~\cite{GAP4}\xspace}
\newcommand{\Semigroups}{\textsc{Semigroups}~\cite{Mitchell2026aa}\xspace}
\newcommand{\T}{\mathcal{T}}
\newcommand{\PT}{\mathcal{PT}}

\setlist{itemsep=0em}
\title{Short presentations for transformation monoids}
\author{Thomas Aird}
\address{The University of Manchester,
  Alan Turing Building,
  Oxford Rd,
  Manchester,
  M13 9PL,
  UK; and
Heilbronn Institute for Mathematical Research, Bristol, UK.}
\email{thomas.aird@manchester.ac.uk}

\author{James D. Mitchell}
\address{Mathematical Institute, North Haugh, St Andrews, Fife, Scotland, KY16
9SS}
\email{jdm3@st-andrews.ac.uk}

\author{Murray T. Whyte}
\address{Mathematical Institute, North Haugh, St Andrews, Fife, Scotland, KY16
9SS}
\email{mw231@st-andrews.ac.uk}

\begin{document}

\begin{abstract}
  By the theorems of Cayley and Vagner-Preston, the full transformation
  monoids and the symmetric inverse monoids play analogous roles in the theory
  of monoids and inverse monoids, as the symmetric groups do in the theory of
  groups. Every presentation for the finite full transformation monoids $T_n$,
  symmetric inverse monoids $I_n$, and partial transformation monoids $PT_n$
  contains a monoid presentation for the symmetric group. In this paper we show
  that the number of relations required, in addition to those for the symmetric
  group, for each of these monoids is at least $4$, $3$, and $8$, respectively.
  We also give presentations for: $T_n$ with $4$ additional relations when
  $n\geq 7$; for $I_n$ with $3$ additional relations for all $n \geq 3$; and
  for $PT_n$ with $8$ additional relations for all $n\geq 7$. The presentations
  for $T_n$ and $I_n$ answer open problems in the literature.
\end{abstract}

\maketitle
\tableofcontents

\section{Introduction}\label{section-intro}

The idea of defining semigroups, monoids, and groups by presentations is almost
as old as the study of these objects themselves and the accompanying literature
is vast. Prominent early examples include~\cites{Aizenstat1958aa,
Carmichael1937aa, Moore1897aa, Popova1961aa, Sutov1960aa}; and recent
examples, among many others, include~\cites{Akita2024aa, Chalk2024aa,
  Chinyere2023aa, Cihan2023aa, Clark2023aa, Curien2023aa, Dimitrova2023aa,
  Elias2023aa, Kobayashi2023aa, Namanya2023aa, Quick2024aa, Reese2023aa,
Yayi2026aa, Wang2023aa}. The question of optimising, in several
different senses,
presentations for particular classes of groups, monoids, or semigroups, has
also been extensively considered in the literature. These notions of optimising
presentations usually relate to minimising something: the difference between
the number of generators and relations~\cite{Abdolzadeh2013aa, Campbell2004aa,
Conder2006aa, Stoytchev2022aa}; number of
generators~\cite{Gomes1987aa,Gray2013aa}; number of
relations~\cite{Babai1997aa}; or some notion of the length of a
presentation~\cite{Babai1984aa, Bray2011aa, Guralnick2008aa, Lubotzky2010aa}.
In this paper we are primarily concerned with the second of these notions: the
number of relations. We will also make some comments about the lengths of
various presentations $\P$, where the \defn{length} is defined to be the sum of
the number of generators and the lengths of all of the words occurring in
relations (or relators); denoted by $|\P|$.

In this paper, we consider presentations for three classical finite
transformation monoids: the symmetric inverse monoids $I_n$; the full
transformation monoids $T_n$; and the partial transformation monoids $PT_n$,
where $n\in \N$. The value $n$ is referred to as the \defn{degree} of these
monoids. The group of units of all of these monoids is the finite symmetric
group $S_n$, where $n$ is the degree of the monoid; we also refer to $n$ as the
degree of $S_n$. It is not difficult to show that any presentation $\P$ for any
of these monoids must contain a monoid presentation for the corresponding
symmetric group. The symmetric groups, the full transformation monoids, and the
symmetric inverse monoids play analogous roles in the classes of groups,
monoids, and inverse monoids respectively. By Cayley's Theorem~\cite[Theorem
1.1.2]{Howie1995aa} every monoid is isomorphic to a submonoid of some full
transformation monoid; and by the Vagner-Preston Representation
Theorem~\cite[Theorem 5.1.7]{Howie1995aa}, every inverse monoid is isomorphic to
an inverse submonoid of some symmetric inverse monoid. A definition of these
monoids and their elements can be found in \cref{section-prelims}.

Presentations for the finite symmetric groups also have a long history, perhaps
starting with Moore~\cite{Moore1897aa} and Burnside and
Miller~\cite{Burnside2012aa} in 1897, through
Carmichael~\cite{Carmichael1937aa} in 1937, Coxeter and
Moser~\cite{Coxeter1980aa} in 1958 into the present day in~\cite{Bray2011aa,
Guralnick2008aa, Lubotzky2010aa}. Because of their particular relevance here,
we will explicitly state a number of these presentations. Moore's presentation
for the symmetric group $S_n$ of degree $n\geq 4$ from~\cite{Moore1897aa} is:
\begin{equation}\label{eq-moore}
  \presn{a, b}{a^{2},\quad b ^ {n},\quad  {(ba)} ^ {n - 1},\quad
    {(ab^{n - 1} a b)}^{3},\quad
  {(a b^{n - j} a b^j)}^{2}\quad (2 \leq j < n - 1)},
\end{equation}
where $a$ represents the transposition $(1, 2)$, and $b$ represents the
$n$-cycle $(1, 2, \ldots, n)$.
The number of relations in Moore's presentation is $n + 1$; and the
length of the presentation is $O(n^2)$.
Carmichael's presentation for $S_n$ with $n\geq 4$
from~\cite[p169]{Carmichael1937aa} in 1937 is:
\begin{equation}\label{eq-carmichael}
  \begin{aligned}
    \langle
    a_2, \ldots, a_{n}
    & \mid
    a_2 ^ {2}, \ldots, a_n^2, {(a_2a_3)}^3, \ldots,
    {(a_{n - 1}a_n)}^3, {(a_{n}a_{2})}^3, \\
    &\ \  {(a_{i}a_{i+1}a_{i}a_{j})}^2\quad (2\leq i, j\leq n,\
    j\not\in \{i, i + 1\},\ a_{n + 1} = a_2)\rangle,
  \end{aligned}
\end{equation}
where $a_i$ represents the transposition $(1, i)$ for every $i$;
see also~\cite[Item 9.5.2]{Ganyushkin2009aa}.
The number
of relations in, and the length of, Carmichael's presentation are both $O(n^2)$.
Both of these presentations are monoid presentations, since they do not contain
inverses of any generators. As suggested in~\cite{Bray2011aa}, Moore's
presentation from~\eqref{eq-moore} can be modified to have $O(n)$ relations
and length. For example, one such modification of Moore's presentation
is:
\begin{equation}\label{eq-moore-reduced}
  \presn{a, b, c_2, \ldots, c_{n - 2}}{a^{2},\quad b ^ {n},\quad  {(ba)} ^ {n -
    1},\quad {(ab^{n - 1} a b)}^{3},\quad c_2 b ^ {-2}, \quad c_{j
    +1}^{-1}c_{j}b, \quad {(a
  c_{n-j} a c_j)}^{2}\quad (2 \leq j < n - 1)}.
\end{equation}
Although this presentation is not a monoid presentation, it can be converted
into one by replacing, for example, the relator $c_2b^{-2}$ by the relation
$c_2 = b^2$. It is shown in~\cite[Theorem 1.1]{Bray2011aa} that $S_n$
has a presentation with $O(\log n)$ relators and length $O(n)$\footnote{In
  Theorem 1.1 from~\cite{Bray2011aa} it is written that the length of the
  presentation is $O(\log ^2 n)$ but this uses a different definition than
that used here, of the length of a presentation.}. The presentation
in~\cite{Bray2011aa} is not a monoid presentation, but by replacing the
inverse of a generator $a$ wherever it occurs by $a^{m - 1}$ where $m$ is the
order of $a$, we can obtain a monoid presentation for $S_n$ with the same
number of relations, and length $O(n)$. Further presentations of $S_n$ and
an analysis of their numbers of relators and length are given
in~\cite{Lubotzky2010aa}.

Given that a monoid presentation for the symmetric group must be contained in
every presentation for a full transformation monoid, a symmetric inverse
monoid, or a partial transformation monoid, and that presentations for the
symmetric group have been well-studied, we will not attempt to further optimise
these presentations. Instead we focus on minimising the number of additional
relations, i.e.\ those not defining the symmetric group. We refer to these
relations collectively as \defn{non-$S_n$ relations}. In the main theorems of
this paper, we will not explicitly include specific generators or relations
defining the symmetric groups. As a consequence, when the non-$S_n$ relations
in a presentation should contain words that define elements of the symmetric
group, we instead write the corresponding element of the symmetric group in
disjoint cycle notation.

The following presentation for $I_n$ has $5$ non-$S_n$ relations; see
\cref{section-prelims} for any undefined notation.

\begin{theorem}[Sutov~\cite{Sutov1960aa} 1960,
    Popova~\cite{Popova1961aa} 1961, East~\cite{East2007aa}
  2007]\label{theorem-In-five-rels}
  Suppose that $n\in \N$, that $n \geq 4$, and that $\presn{A}{R}$ is any
  monoid presentation for the symmetric group $S_n$ of degree $n$. Then the
  symmetric inverse monoid $I_n$ is defined by the presentation with generators
  $A\cup \{\vep\}$ where $\vep\not\in A$, and relations $R$ together with:
  \begin{multicols}{2}
    \begin{enumerate}[label=\rm (I\arabic*), ref=\rm I\arabic*]
      \item\label{rel-I1}
        $\vep^2 = \vep$;
      \item\label{rel-I2}
        $(2, 3) \vep = \vep (2, 3)$;
      \item\label{rel-I3}
        $(2, 3, \ldots, n) \vep = \vep(2, 3, \ldots, n)$;\columnbreak%
      \item\label{rel-I4}
        $\vep  (1, 2)   \vep  (1, 2)  =  \vep(1, 2)\vep$;
      \item\label{rel-I5}
        $(1, 2)   \vep  (1, 2)   \vep =  \vep(1, 2)\vep$;
      \item[\vspace{\fill}]
    \end{enumerate}
  \end{multicols}
  \noindent (where $(2, 3)$, $(2, 3, \ldots, n)$ and $(1, 2)$ denote any words
    in $A ^*$ representing these permutations, and $\vep$ represents the rank $n
  - 1$ idempotent with $1\notin \dom(\vep)$).
\end{theorem}

Note that the relations~\ref{rel-I1},~\ref{rel-I2}, and~\ref{rel-I3} represent
elements in $I_n$ with rank $n - 1$, and~\ref{rel-I4} and~\ref{rel-I5}
represent elements with rank $n - 2$.

The non-$S_n$ relations in the presentation from \cref{theorem-In-five-rels}
have length
\[15 + 2 |(2, 3)| + 2 |(2,3, \ldots, n)| + 6 |(1, 2)|,\]
where $|\sigma|$ denotes the length of the shortest word in $A^*$ representing
$\sigma\in S_n$. If $A$ is the generating set consisting of $(1, 2)$ and $(1,
2, \ldots, n)$, then it is straightforward to check that
\[
  |(2, 3)| = n + 1,\quad |(1, 2)| = 1,\quad |(2, 3, \ldots, n)| = 2.
\]
It follows that there is a presentation for $I_n$ with number, and length, of
relations $O(n)$ (i.e.\ using the presentation
  in~\eqref{eq-moore-reduced} for the
symmetric group in \cref{theorem-In-five-rels}).
Similarly, since the presentation with $O(\log n)$ relations and length $O(n)$
defining $S_n$ from~\cite{Bray2011aa} has $(1, 2)$ and $(1, 2, \ldots, n)$ among
its generators, it follows that $I_n$ has a presentation with $O(\log n)$
relations and length $O(n)$.

It is also possible to prove that none of the non-$S_n$ relations in
\cref{theorem-In-five-rels} is unnecessary in the following sense. We say that a
relation $(u, v)$ in a presentation $\presn{A}{R}$ is \defn{redundant} if the
monoid defined by the presentation $\presn{A}{R\setminus \{(u, v)\}}$ is
isomorphic to the monoid defined by $\presn{A}{R}$. A relation that is not
redundant is \defn{irredundant}. If every relation belonging to
$R'\subseteq R$ is
irredundant, we refer to $R'$ as being \defn{irredundant}.
It can be shown that the
relations~\ref{rel-I1},~\ref{rel-I2},~\ref{rel-I3},~\ref{rel-I4},
and~\ref{rel-I5} are irredundant. However, since this is not
the purpose of this paper, we omit the details.

There is a large presentation for $T_n$ from the 1970s, due to Iwahori and
Iwahori~\cite{iwahori}, where the number of non-$S_n$ relations is
$O(n^4)$. Given the large number of relations in this presentation, we will not
reproduce it here. Pre-dating Iwahori and Iwahori's presentation by nearly two
decades, is the considerably smaller presentation for $T_n$ due to Aizenstat,
with only seven non-$S_n$ relations.

\begin{theorem}[Aizenstat, 1958~\cite{Aizenstat1958aa}]\label{theorem-aizenstat}
  Suppose that $n\in \N$, that $n \geq 4$, and that $\presn{A}{R}$ is
  any monoid presentation for the symmetric group $S_n$ of degree $n$. Then the
  full transformation monoid $T_n$ is defined by the presentation with
  generators $A\cup \{\ve\}$ with $\ve\notin A$, and relations $R$
  together with:
  \setlength{\columnsep}{-2cm}
  \begin{multicols}{2}
    \begin{enumerate}[leftmargin=1em, label=\rm (T\arabic*), ref=\rm T\arabic*]
      \item\label{rel-T1}
        $\ve  (1, 3) \ve  (1, 3) = \ve$;
      \item\label{rel-T2}
        $(1, 2)  \ve = \ve$;
      \item\label{rel-T3}
        $(3, 4) \ve = \ve  (3, 4)$;
      \item\label{rel-T4}
        $(3, 4, \ldots, n)  \ve = \ve  (3, 4, \ldots, n)$;\columnbreak%
      \item\label{rel-T6}
        $\ve (2, 3) \ve (2, 3) = \ve (2, 3) \ve$;
      \item\label{rel-T5}
        $(2, 3) \ve (2, 3) \ve = \ve (2, 3) \ve$;
      \item\label{rel-T7}
        $(1, n)(2, 3) \ve  (1, n)(2, 3) \ve
        = \ve (1, n)(2, 3)  \ve  (1, n)(2, 3)$;
    \end{enumerate}
  \end{multicols}
  \noindent (where the permutations written in disjoint cycle notation denote
    any words in $A^*$ representing these permutations, and $\ve$ represents the
  rank $n - 1$ idempotent mapping $2$ to $1$).
\end{theorem}

Note that in~\cite{Aizenstat1958aa} the relation $\ve(1, n)\ve(1, n) = \ve$ is
used instead of~\ref{rel-T1}. However,
if $\alpha = (3, 4, \ldots, n)$, then, using the relations $R$ for
$S_n$,~\ref{rel-T1}, and~\ref{rel-T4},
\begin{align*}
  \ve & = \alpha\ve \alpha^{-1} \\
  & = \alpha\ve(1, 3)\ve (1, 3)\alpha^{-1}\\
  & =\alpha\ve\alpha^{-1} \alpha(1, 3)\alpha^{-1}\alpha\ve
  \alpha^{-1}\alpha(1, 3)\alpha^{-1}\\
  & = \alpha\ve\alpha^{-1}(1, n)\alpha\ve \alpha^{-1}(1, n)\\
  & = \ve (1, n)\ve(1, n).
\end{align*}
Hence the relation $\ve (1, n) \ve (1, n)$ from Aizenstat's original
presentation
from~\cite{Aizenstat1958aa}
holds in the monoid defined by the
presentation in \cref{theorem-aizenstat}. Conversely, the relation
$\ve(1, 3)\ve(1, 3) = \ve$ holds in $T_n$ and hence in the monoid
defined by Aizenstat's original presentation.

The relations~\ref{rel-T1},~\ref{rel-T2},~\ref{rel-T3}, and~\ref{rel-T4}
represent transformations with rank $n - 1$, and the
relations~\ref{rel-T6},~\ref{rel-T5}, and~\ref{rel-T7} represent
transformations with rank $n - 2$.
Again it is possible to show that none of the non-$S_n$ relations in
\cref{theorem-full-transf-main} are redundant, but this is not our focus here,
and so we omit the details.

The non-$S_n$ relations in the presentation from \cref{theorem-aizenstat} have
length:
\begin{equation}\label{eq-length-Tn}
  21 + 2 |(1, 3)| + |(1, 2)| + 2 |(3, 4)| + 2 |(3, 4, \ldots, n)|
  + 10 |(2, 3)| + 4 |(1, n)|.
\end{equation}
It is possible to show that the lengths of each
of permutations appearing in~\eqref{eq-length-Tn}, with respect to the
generators $(1, 2)$ and $(1, 2, \ldots, n)$, is $O(n)$. Hence using the
presentation for $S_n$ from~\cite{Bray2011aa} and a similar argument to that
given above for the symmetric inverse monoid $I_n$, it is possible to show
that there is a presentation for $T_n$ with $O(\log n)$ relations and length
$O(n)$.

Sutov gave a presentation for the partial transformation monoids $PT_n$
in~\cite{Sutov1960aa}, consisting of a presentation for the symmetric inverse
monoid $I_n$, a presentation for the full transformation monoid $T_n$, and four
more relations. Using the presentations from \cref{theorem-In-five-rels} and
\cref{theorem-aizenstat}, for example, Sutov's presentation for $PT_n$ has $16$
relations in addition to those defining the symmetric group.
In~\cite{East2007ab}, East gives a presentation for $PT_n$ consisting of a
presentation for $S_n$ and twelve non-$S_n$ relations. The presentation given
below is a slight modification of East's original presentation, which better
suits our purposes here.

\begin{theorem}[cf.~\cite{East2007ab}, Theorem 30]\label{theorem-PTn-east}
  Suppose that $n\in \N$, that $n \geq 4$, and that $\presn{A}{R}$ is
  any monoid presentation for the symmetric group $S_n$ of degree $n$. Then
  the partial transformation monoid $PT_n$ is defined by
  the presentation with generators $A\cup \{\ve, \vep\}$, and relations $R$
  together with the relations:
  \setlength{\columnsep}{-2cm}
  \begin{multicols}{2}
    \begin{enumerate}[leftmargin=1em, label=\rm (P\arabic*), ref=\rm P\arabic*]
      \item [\rm(\ref{rel-I2})]
        $(2, 3)  \vep = \vep (2, 3)$;
      \item [\rm(\ref{rel-I3})]
        $(2, 3, \ldots, n)  \vep = \vep (2, 3, \ldots, n)$;
      \item [\rm(\ref{rel-I4})]
        $\vep (1, 2)  \vep (1, 2)  = \vep (1, 2)  \vep$;
      \item [\rm(\ref{rel-T2})]
        $(1, 2) \ve = \ve$;
      \item [\rm(\ref{rel-T3})]
        $(3, 4)\ve = \ve (3, 4)$;
      \item [\rm(\ref{rel-T4})]
        $(3, 4, \ldots, n) \ve = \ve (3, 4, \ldots, n)$;\columnbreak%
      \item [\rm(\ref{rel-T7})]
        $(1, n)(2, 3)\ve (1, n)(2, 3) \ve
        = \ve (1, n)(2, 3) \ve (1, n)(2, 3)$;
      \item\label{rel-P1}
        $\ve (1, 2) \vep (1, 2)  = \ve$;
      \item\label{rel-P2}
        $\vep (1, 2)  \ve (1, 2)  = \vep$;
      \item\label{rel-P3}
        $\ve \vep = \vep (1, 2)  \vep$;
      \item\label{rel-P4}
        $(1, 3)  \ve (1, 2, 3)  \ve = \ve (1, 2, 3)  \ve$;
      \item\label{rel-P5}
        $(1, 3)  \vep (1, 3)  \ve = \ve (1, 3)  \vep (1, 3)$;
    \end{enumerate}
  \end{multicols}
  \noindent (where the permutations written in disjoint cycle notation denote
    any words in $A^*$ representing these permutations, $\ve$
    represents the rank $n - 1$ idempotent full transformation
    mapping $2$ to $1$, and $\vep$ represents the rank $n - 1$
  idempotent partial bijection with $1 \notin \operatorname{dom} \vep$).
\end{theorem}

The
relations~\ref{rel-I2},~\ref{rel-I3},~\ref{rel-T2},~\ref{rel-T3},~\ref{rel-T4},~\ref{rel-P1},
and~\ref{rel-P2} represent partial transformations
with rank $n - 1$, and~\ref{rel-I4},~\ref{rel-T7},~\ref{rel-P3},~\ref{rel-P4},
and~\ref{rel-P5} represent partial transformations with rank $n - 2$. Similar
to \cref{theorem-In-five-rels,theorem-aizenstat}, it is possible to show that
the non-$S_n$ relations in \cref{theorem-PTn-east} are irredundant, but we
omit the details.

The length of the non-$S_n$ relations in the presentation given in
\cref{theorem-PTn-east} is:
\[
  36 + 9 |(1, 2)| + 2 |(3, 4)| + 2 |(3, 4, \ldots, n)| + 4 |(1
  \ n)| + 6 |(2,
  3)| + 2 |(2, 3, \ldots, n)| + 5 |(1, 3)| + 2 |(1, 2, 3)|.
\]
As a corollary of \cref{theorem-PTn-east} and using a similar
argument to that used above for $I_n$ and $T_n$, it can be shown that there is a
presentation for $PT_n$ with $O(\log n)$ relations and length $O(n)$.

In this paper we show that the number of relations required, in addition to
those for the symmetric group, for each of the symmetric inverse, full
transformation, and partial transformation monoids is at least $3$,
$4$, and $8$,
respectively.

The main theorems of this paper are the following.

\begin{theorem}[\textbf{Symmetric inverse monoid}]\label{theorem-In-main}
  Suppose that $n\in \N$, that $n \geq 4$, and that $\presn{A}{R}$ is
  any monoid presentation for the symmetric group $S_n$ of degree $n$.
  Then the symmetric inverse monoid $I_n$ is defined by the
  presentation $\mathcal{I}$ with generators $A\cup \{\vep\}$ where
  $\vep\not\in A$, and relations $R$ and the following $3$ relations:
  \setlength{\columnsep}{-2cm}
  \begin{multicols}{2}
    \begin{enumerate}[leftmargin=1em, label=\rm (I\arabic*), ref=\rm I\arabic*]
      \item [\rm(\ref{rel-I2})] $(2, 3) \vep = \vep (2, 3)$;
        \addtocounter{enumi}{5}
      \item\label{rel-I6}
        $(2, 3, \ldots, n)\vep = \vep^2 (2, 3, \ldots, n) $;
      \item\label{rel-I7}
        $(1, 2) \vep(1, 2)\vep(1, 2)\vep (1, 2) = \vep (1, 2) \vep$;
    \end{enumerate}
  \end{multicols}
  \noindent (where $(2, 3)$, $(2, 3, \ldots, n)$ and $(1, 2)$ denote any words
    in $A ^*$ representing these permutations and $\vep$ represents
  the rank $n - 1$ idempotent with $1 \notin \dom(\vep)$).

  Furthermore, every presentation for $I_n$, $n\geq 3$, requires at least $3$
  non-$S_n$ relations.
\end{theorem}

The relations~\ref{rel-I2} and~\ref{rel-I6} represent partial permutations of
rank $n - 1$, and the relation~\ref{rel-I7} represents an element of rank $n -
2$.
When showing that every presentation for $I_n$, $n \geq 3$ requires at
least $3$ non-$S_n$ relations, we show that every such presentation requires at
least one rank $n - 2$ relation (\cref{lemma-In-at-least-1-n-2}), which proves
the main theorem in~\cite{Yayi2026aa} for the symmetric inverse monoid in the
case that ``$m=n$''.

That the non-$S_n$ relations in \cref{theorem-In-main} are irredundant
follows immediately from the last part of the theorem, since the minimum number
of such relations is $3$.
We consider the cases when $n\in \{1, 2, 3\}$ in
\cref{section-In-leq-3}.

The question of the minimum number of non-$S_n$ relations required to
define $I_n$ was posed in~\cite[Open Problem 4]{ruskuc}, and~\cite[Remark
34]{East2007aa}.\ \cref{theorem-In-main} shows that the answer to these problems
is three.

\begin{theorem}[\textbf{Full transformation
  monoid}]\label{theorem-full-transf-main}
  Suppose that $n\in \N$, that $n\geq 7$ and that $\presn{A}{R}$
  is any monoid presentation for the symmetric group $S_n$ of degree $n$. Then
  $T_n$ is defined by the presentation with generating set $A\cup \{\zeta\}$
  where $\zeta\not\in A$ and relations consisting of
  $R$ and the following $4$ relations:
  \setlength{\columnsep}{2cm}
  \begin{multicols}{2}
    \begin{enumerate}[leftmargin=1em, label=\rm (T\arabic*), ref=\rm T\arabic*]
        \addtocounter{enumi}{7}
      \item [\rm (\ref{rel-T7})]
        $(1, n)(2, 3) \ve  (1, n)(2, 3)  \ve = \ve (1, n)(2, 3)  \ve
        (1, n)(2, 3)$;\quad
      \item\label{rel-T8}
        $(2, 3)  \ve  (2, 3)  \ve(2, 3)  \ve(2, 3) = \ve  (2, 3)   \ve$;
    \end{enumerate}
    \begin{enumerate}[leftmargin=1em, label=\rm (T\textgreek{\alph{enumi}}),
        ref=T\textgreek{\alph{enumi}}]
      \item\label{rel-T-alpha}
        $\alpha ^{-1} \ve  \alpha = \ve  (1, 3)  \ve  (1, 3)$;
      \item\label{rel-T-beta}
        $\beta^{-1} \ve \beta =(1, 2)  \ve (1, 3)   \ve (1, 3)$;
    \end{enumerate}
  \end{multicols}
  \noindent where $\alpha = (3, 4)$ and $\beta = (3,\ldots, n)$ if $n$ is odd;
  and $\alpha = (3,\ldots, n)$ and $\beta = (3, 7, 6, 4, 5)$ if $n$ is even.
  (The permutations written in disjoint cycle notation denote any words in
    $A^*$ representing these permutations, and $\ve$ represents the rank $n - 1$
  idempotent mapping $2$ to $1$).

  Furthermore, every presentation for $T_n$, $n\geq 4$, requires at least $4$
  non-$S_n$ relations.
\end{theorem}

In the presentations from \cref{theorem-full-transf-main},
relations~\ref{rel-T7} and~\ref{rel-T8} represent transformations with rank $n
- 2$, and the other non-$S_n$ relations represent transformations with rank $n
- 1$.
When showing that every presentation for $T_n$, $n \geq 4$ requires at
least $4$ non-$S_n$ relations, we show that every such presentation requires at
least two rank $n - 2$ relations (\cref{lemma-Tn-at-least-2-n-2}), which proves
the main theorem in~\cite{Yayi2026aa} for the full transformation monoid in the
case that ``$m=n$''.

Every presentation for $T_n$ when $n = 1, 2, 3, 4, 6$ must contain
at least $0, 2, 3, 4, 4$, respectively, non-$S_n$ relations (see
\cref{section-Tn-leq-4} for presentations meeting this lower bound).
When $n = 5$ it can be shown using the \Semigroups package
for \GAP that the presentation given in \cref{theorem-full-transf-main} defines
$T_5$; the proof of \cref{theorem-full-transf-main} requires $n\geq 7$.

\cref{theorem-full-transf-main}, \cref{align-full-transf-4}, and
\cref{align-full-transf-6} provide an answer of ``four'' to the question
of what the minimum number of non-$S_n$ relations in any presentation for $T_n$
is when $n \geq 4$ (see Open Problem 2 of~\cite{ruskuc}).

\begin{theorem}[\textbf{Partial transformation monoid}]\label{theorem-PTn-main}
  Suppose that $n\in \N$, that $n \geq 7$, and that $\presn{A}{R}$ is any
  monoid presentation for the symmetric group $S_n$ of degree $n$. Then $PT_n$
  is defined by the presentation $\PT$ with generators $A\cup \{\ve, \vep\}$
  where $\ve, \vep\not\in A$ and relations $R$ and the following $8$ relations:
  \setlength{\columnsep}{2cm}
  \begin{multicols}{2}
    \begin{enumerate}[leftmargin=1em, label=\rm (P\arabic*), ref=\rm P\arabic*]
        \addtocounter{enumi}{5}
      \item [\rm (\ref{rel-T7})]
        $(1, n)(2, 3) \ve  (1, n)(2, 3)  \ve = \ve (1, n)(2, 3)  \ve
        (1, n)(2, 3)$;
      \item[\rm (\ref{rel-T8})]
        $(2, 3)  \ve  (2, 3)  \ve(2, 3)  \ve(2, 3) = \ve  (2, 3)   \ve$;
      \item [\rm (\ref{rel-P5})]
        $(1,  3)   \vep  (1,  3)   \ve = \ve  (1,  3)   \vep  (1,  3)$;
      \item\label{rel-P6}
        $(1,  2)   \vep  (1,  2)   \vep  (1,  2) = \ve \vep$;
    \end{enumerate}
    \quad
    \begin{enumerate}[leftmargin=1em, label=\rm (P\textgreek{\alph{enumi}}),
        ref=P\textgreek{\alph{enumi}}]
      \item\label{rel-P-alpha}
        $\alpha ^{-1}\ve \alpha = (1, 2)\ve (1, 2)\vep (1, 2)$;
      \item\label{rel-P-beta}
        $\beta ^{-1}\ve \beta = \ve (1, 2)\vep (1, 2)$;
        \setcounter{enumi}{6}
      \item\label{rel-P-gamma}
        $\gamma^{-1}\vep \gamma = \vep (1, 2)\ve (1, 2)$;
        \setcounter{enumi}{3}
      \item\label{rel-P-delta}
        $\delta^{-1}\vep \delta = \vep (1, 2)\ve (1, 2)$;
    \end{enumerate}
  \end{multicols}
  \noindent
  where $\alpha = (3,\ldots, n)$, $\beta = (4, 6)$, $\gamma = (2, \ldots, n)$,
  and $\delta = (2, 3)(4, 6)$ if $n$ is odd; and $\alpha = (3, 5, 4)$, $\beta =
  (3, \ldots, n)$, $\gamma = (2, 3, 5, 4)$, and $\delta = (2, \ldots, n)$ if
  $n$ is even.
  (The permutations written in disjoint cycle notation denote any words in
    $A^*$ representing these permutations, $\ve$ represents the rank $n - 1$
    idempotent full transformation mapping $2$ to $1$, and $\vep$ represents the
    rank $n - 1$ idempotent partial bijection with $1 \notin \operatorname{dom}
  \vep$).

  Furthermore, if $n \geq 4$, then every presentation for $PT_n$
  requires at least $8$ non-$S_n$ relations.
\end{theorem}

The relations~\ref{rel-P-alpha},~\ref{rel-P-beta},~\ref{rel-P-gamma},
and~\ref{rel-P-delta} represent elements of $PT_n$ with rank $n - 1$,
and~\ref{rel-T7},~\ref{rel-T8},~\ref{rel-P5}, and~\ref{rel-P6} represent
elements with rank $n - 2$.
The proof of the last part of this theorem (see
\cref{lemma-at-least-4-n-2-PTn}) proves the main
theorem in~\cite{Yayi2026aa} for the partial transformation monoid in the case
that ``$m=n$''.

We consider the cases when $n\in \{1, \ldots, 6\}$ in
\cref{section-PTn-leq-4}.

This paper is organised as follows: in~\cref{section-prelims} we introduce some
prerequisite notions that we require in later sections, and we give some
general results about presentations for monoids; in \cref{section-alt-group} we
provide a number of specific generating sets for the alternating groups which
we require to prove our main theorems; in
\cref{section-In,section-Tn,section-PTn} we prove the main theorems
(\cref{theorem-In-main,theorem-full-transf-main,theorem-PTn-main}) for the
symmetric inverse monoid, full transformation monoid, and partial
transformation monoid, respectively.

We conclude this section with an open problem.
As noted above, if any of the presentations for the symmetric group $S_n$, for
large enough $n\in \N$, explicitly stated in the introduction is used in
\cref{theorem-In-five-rels,theorem-aizenstat,theorem-PTn-east,theorem-In-main,theorem-full-transf-main,theorem-PTn-main},
then the non-$S_n$ relations can be chosen to have length $O(n)$. For an
arbitrary presentation $\presn{A}{R}$ for $S_n$, however, the question of the
length of the non-$S_n$ relations in any of the theorems in \cref{section-intro}
is somewhat related to the diameter of the Cayley graph of $S_n$ with respect
to the generating set $A$.  This is seemingly an extremely difficult question;
see for example~\cite{Babai1989aa, Helfgott2011aa}. On the other hand, the
diameter of the Cayley graph only provides an upper bound for the length of the
non-$S_n$ relations in \cref{section-intro}. Determining an upper bound on the
shortest paths in the Cayley graph from the identity to the permutations
appearing in the non-$S_n$ relations might be more tractable. As such we state
the following open problem.

\begin{openproblem}
  If $\P$ is any presentation for the symmetric group $S_n$, then does there
  always exist a presentation for $I_n$, $T_n$, and/or $PT_n$ where the length
  of the non-$S_n$ relations is $O(n)$?
\end{openproblem}

\section{Preliminaries}\label{section-prelims}

In this section, we introduce the relevant preliminary material for the later
sections of the paper.

If $M$ is a monoid (a set with an associative binary operation, and an identity
element), and $\rho \subseteq M\times M$, then $\rho$ is a
\defn{congruence} on $M$ if
$(x, y)\in \rho$ and $s\in M$ imply that $(xs, ys), (sx, sy)\in \rho$. If $R
\subseteq M \times M$, then we denote by $R^{\sharp}$ the least congruence on
$M$ containing $R$. We refer to the minimum size of a generating set for a
monoid $M$ as the \defn{rank} of $M$; and we denote this by $d(M)$. If $M$ is a
monoid and $N$ is a subset of $M$ containing the identity $1_M$ of $M$ and that
is closed under the operation of $M$, then $N$ is a \defn{submonoid} of $M$;
and we write $N\leq M$.

A \defn{partial transformation} of a set $X$ is a function from a subset $Y$ of
$X$ into $X$. We denote a partial transformation $f$ of $X$ by $f: X\to X$, and
we write $(x)f$ for the image of $x\in X$ under $f$. The set $Y$ is referred to
as the \defn{domain} of $f$ and is denoted by $\dom(f)$. Similarly, the
\defn{image} of $f$ is $\im(f) = \set{(x)f\in X}{x\in \dom(f)}$. The
\defn{rank} of $f\in PT_X$ is $|\im(f)|$, which we denote by $\rank(f)$.
The \defn{kernel} of $f$ is the binary relation $\ker(f) = \set{(x, y)\in
\dom(f) \times \dom(f)}{(x)f = (y)f}$. The collection of all partial
transformations of a set $X$ forms a monoid with the operation being the usual
composition of binary relations; this monoid is called the \defn{partial
transformation monoid}, and is denoted by $PT_X$. If $|X| = n \in \N$, then we
will write $PT_n$ instead of $PT_X$. It is routine to verify that $\dom(fg)
\subseteq \dom(f)$, $\im(fg)\subseteq \im(g)$, and $\ker(f) \subseteq \ker(fg)$
for all $f, g\in PT_X$. It is well-known that the ideals of $PT_n$
are $\set{f\in PT_n}{\rank(f) \leq r}$ for every $r\in \{0, \ldots, n\}$.

A \defn{transformation} of $X$ is a partial transformation $f: X\to X$ such
that $\dom(f) = X$. The \defn{full transformation monoid} $T_X$ is the
submonoid of $PT_X$ consisting of the transformations of $X$. Note that
composition of binary relations coincides with composition of functions on
$T_X$.  A \defn{partial permutation} of $X$ is an injective partial
transformation $f: X \to X$. The \defn{symmetric inverse monoid} $I_X$ is the
submonoid of $PT_X$ consisting of all partial permutations of $X$. A bijective
transformation is a \defn{permutation} and the group of all permutations of $X$
is denoted by $S_{X}$. If $n\in\N$ and $X = \{1, \ldots, n\}$, then, as we did
for $PT_n$, we write $T_n$, $I_n$, and $S_n$ instead of $T_X$, $I_X$, or $S_X$.
Similarly, if $X$ is finite, then we denote the alternating group on $X$ by
$A_X$ or $A_n$ as appropriate.

An \defn{alphabet} is any non-empty set, and a \defn{letter} is an element of
an alphabet. If $A$ is an alphabet, then the \defn{free monoid} $A^*$ consists
of all finite sequences $(a_1, \ldots, a_n)$ of letters $a_i\in A$, where $n
\geq 0$, including the empty sequence $\varepsilon$. Such a sequence is
referred to as a \defn{word}. We write
$a_1\cdots a_n$ instead of $(a_1, \ldots, a_n)\in A ^*$. The
operation on $A ^*$ that
makes it a monoid is concatenation of words.

A \defn{presentation} is a pair $\P = \presn{A}{R}$ where $A$ is an alphabet,
and $R\subseteq A ^ * \times A ^*$. The monoid defined by $\P$ is the quotient
monoid $A ^ */ R ^{\sharp}$. In this paper, we are solely concerned with finite
presentations $\presn{A}{R}$, i.e.\ those where both $A$ and $R$ are finite. We
will follow the convention that omits all of the curly braces and writes
relations $(u, v)$ as $u=v$ in a presentation. For example, if $m, r\in \N$,
then we write
\[
  \presn{a}{a ^ {m + r} = a ^ m}
\]
instead of
\[
  \presn{\{a\}}{\{(a ^ {m + r}, a ^ m)\}}.
\]
If $u, v\in A ^ *$ and $R\subseteq A ^ *\times A ^ *$, then outside of a
presentation we may also write $u = v$ to mean that $(u, v)\in R$ if it is
possible to do this without ambiguity. On the other hand, if ambiguity might
arise, we will write $(u, v)\in R$ instead. An \defn{elementary sequence} from
$x\in A^*$ to $y\in A ^ *$ with respect to $R\subseteq A ^ *\times A ^ *$ is a
finite sequence $w_1 = x, w_2, \ldots, w_n = y$ such that $w_i = p_i
u_i q_i$ and
$w_{i + 1} = p_{i} v_i q_{i}$ for some $p_i, q_i\in A ^ *$, and
$(u_i, v_i)\in R$
or $(v_i, u_i)\in R$ for all $i$. A pair $(x, y)\in A ^ * \times A^*$ belongs
to $R ^ {\sharp}$ if and only if there exists an elementary sequence from $x$
to $y$ with respect to $R$. If $u, v\in A ^ *$, then we will say that $u = v$
holds in a presentation $\mathcal{P} = \presn{A}{R}$ if there exists an
elementary sequence from $u$ to $v$ with respect to the relations $R$ of $\P$.

Suppose that $M$ denotes one of $I_n$, $T_n$, or $PT_n$, and suppose that $\P =
\presn{A}{R}$ is a presentation defining $M$. We will find it useful to refer
to a word $w\in A^*$ as having \defn{rank $r$} if $w$ represents an element of
$M$ with rank $r$ where $0 \leq r\leq n$. If $(u, v)\in R$, then $u$ and $v$
represent the same element of $M$, and so we may also refer to the \defn{rank
of $(u, v)$ or $u = v$}.

Suppose that $\phi: A ^ *\to M$ is any surjective homomorphism with kernel
equal to $R ^ {\sharp}$. If $w\in A ^ *$ and $w'\in A ^ *$ is the longest
prefix of $w$ such that $(w')\phi\in S_n\leq M$, then we refer to
$(w')\phi$ as the
\defn{leading permutation} of $w$. Since $(\varepsilon)\phi \in S_n$, every word
in $A ^ *$ has a leading permutation. If $w\in A ^ *$ is such that
$(w)\phi\not\in S_n$, then the letter immediately after the leading permutation
of $w$ will be referred to as the \defn{leading non-permutation} of $w$. This
letter must exist by the assumption that $(w)\phi\not\in S_n$.

We conclude this section with a number of straightforward observations that we
will nonetheless find useful later on.

Let $M$ be a monoid and let $A$ and $B$ be generating sets for $M$. If
$\phi_A:A^* \to M$  and $\phi_B:B^* \to M$ are the natural surjective
homomorphisms, then we say that words $u \in A^*$ and $v \in B^*$
\defn{represent the same element of $M$} whenever $(u)\phi_A = (v)\phi_B$ in
$M$. For each $a \in A$, we choose a fixed word $(a)\zeta_{A, B}\in B ^ *$
such that $a$ and $(a)\zeta_{A, B}$ represent the same element of
$M$. We also denote by $\zeta_{A, B}$ the unique homomorphism from $A ^ *$ to
$B ^ *$ extending $a \mapsto (a)\zeta_{A, B}$. We refer to $\zeta_{A, B}$ as a
\defn{change of alphabet} mapping from $A$ to $B$. For every $w =a_1 a_2 \cdots
a_k \in A^*$,
\begin{equation}
  (w)\phi_A = (a_1)\phi_A \cdots (a_k)\phi_A =
  (a_1)\zeta_{A, B}\phi_B \cdots (a_k)\zeta_{A, B}\phi_B = (a_1\cdots
  a_k)\zeta_{A, B}\phi_B
  = (w)\zeta_{A, B}\phi_B.
\end{equation}
\begin{figure}
  \centering
  \begin{tikzcd}[row sep=large, column sep=large]
    A \arrow[d] \arrow[r] & M & B \arrow[d] \arrow[l]\\
    A^* \arrow[ur, "\phi_A"] \arrow[rr, bend left=10, "\zeta_{A, B}"] & &
    B^* \arrow[ul, "\phi_B", swap] \arrow[ll, bend left=10, "\zeta_{B, A}"]
  \end{tikzcd}
  \caption{The functions from
  \cref{lemma-pres-change-gs}.}\label{figure-change-of-alphabet}
\end{figure}
Hence the words $w=a_1\cdots a_k \in A^*$ and $(a_1)\zeta_{A,
B}\cdots (a_k)\zeta_{A, B}$ represent the same element of $M$.

Change of alphabet mappings give us a useful way of obtaining a presentation
for a monoid from another presentation, in terms of a different
generating set;
see \cref{figure-change-of-alphabet}.

\begin{lemma}[cf. Theorem 9.1.4
  in~\cite{Ganyushkin2009aa}]\label{lemma-pres-change-gs}
  Let $M$ be a monoid defined by the presentation $\presn{A}{R}$,
  and let $B$ be a generating set for $M$. Then $M$ is defined by the
  presentation $\presn{B}{R'}$  where
  \[
    R' = \set{((u)\zeta_{A, B}, (v)\zeta_{A, B})\in B^*\times
    B^*}{(u, v) \in R} \cup
    \set{(b,(b)\zeta_{B, A}\zeta_{A, B})\in B^*\times B^*}{b \in B}.
  \]
\end{lemma}

For example, consider the monoid $M$ defined by $\presn{ a}{ a^6 =
\varepsilon}$, which is isomorphic to the cyclic group of order six. Suppose
that $A = \{a\}$ and $B = \{b, c\}$ is a generating set with $b =
a^2$ and $c = a^3$ in $M$.
Then $B$ is also a generating set for $M$ and so we may
choose homomorphisms ${\zeta_{A, B}: \{a\}^* \to \{b, c\}^*}$ and
${\zeta_{B,
  A}: \{b, c\}^*
\to \{a\}^*}$ so that $(a)\zeta_{A, B} = b^2 c$, $(b)\zeta_{B,
A} = a^2$ and
$(c)\zeta_{B, A} = a^3$. \cref{lemma-pres-change-gs} implies
that $\presn{b,
c}{(b^2 c)^6 = \varepsilon, (b^2 c)^2 = b, (b^2 c)^3 = c}$ is a
presentation for $M$.

Roughly speaking, the following trivial corollary of
\cref{lemma-pres-change-gs}
will allow us to answer extremal questions about the number of
relations in any
presentation for $M$.

\begin{corollary}\label{cor-presentation-subsets}
  Let $M$ be a monoid defined by the presentation $\presn{A}{R}$,
  and let $B$ be a generating set for $M$ with $B \subseteq A$. Then there is a
  presentation $\presn{B}{R'}$ defining $M$ where $|R'| = |R|$.
\end{corollary}
\begin{proof}
  If $\zeta_{B, A}: B^ * \to A ^ *$ is the identity, and $\zeta_{A, B}: A ^ *
  \to B ^ *$ is any homomorphism whose restriction to $B^*$ is the identity,
  then we may
  apply \cref{lemma-pres-change-gs} to obtain a presentation $\presn{B}{Q}$
  defining $M$. The relations in $Q$ of the form $b = (b)\zeta_{B, A}\zeta_{A,
  B}$ are trivial (i.e.\ the left and right hand sides are identical
  words), and hence can be removed yielding $R'$ with $|R'| = |R|$.
\end{proof}

Another immediate corollary of \cref{lemma-pres-change-gs} is the following
(recalling that $d(M)$ is the minimum possible size of a generating set for
$M$). This is similar to Lemma 2.1 of~\cite{Guralnick2008aa}.

\begin{corollary}[Lemma 2.1
  in~\cite{Guralnick2008aa}]\label{lemma-GKKL-analogue}
  Let $M$ be a monoid defined by the monoid presentation
  $\presn{A}{R}$. Then there is a presentation $\presn{B}{R'}$ for $M$,
  where $| B | = d(M)$ and $| R' | = | R | + d(M)$.
\end{corollary}

The following straightforward lemma will be used repeatedly in
later sections,
and so we record it here.

\begin{lemma}\label{lemma-higher-ranks}
  Suppose that $M$ is a monoid defined by the monoid presentation
  $\presn{A}{R}$ and suppose that $\phi_A: A ^ *\to M$ is the corresponding
  surjective homomorphism. If $w\in A ^*$, then no elementary sequence with
  respect to $R$ containing $w$ contains any word $w'\in A ^ *$ such that
  $M\cdot (w')\phi_A\cdot M = \set{m\cdot (w')\phi_A\cdot n}{m, n\in
  M} \subsetneq
  M\cdot (w)\phi_A \cdot M$.
\end{lemma}

If $M$ in \cref{lemma-higher-ranks} is any of the monoids $I_n$, $T_n$, or
$PT_n$ and $f, g\in M$, then $M\cdot f \cdot M \subsetneq M\cdot
g \cdot M$ if
and only if $\rank(f) < \rank(g)$.

One application of \cref{lemma-higher-ranks} is the following corollary, which
we will apply to the ideals $M\setminus S_n$ when $M$ is any of $I_n$, $T_n$,
or $PT_n$.

\begin{corollary}\label{lemma-complement-ideal-presentation}
  If $M$ is a monoid defined by the presentation $\presn{A}{R}$, and $I$
  is an ideal of $M$ such that $M\setminus I$ is a subsemigroup, then there
  exists $A'\subseteq A$ and $R'\subseteq R$ such that $\presn{A'}{R'}$
  is a presentation for $M\setminus I$.
\end{corollary}

The next lemma is somewhat similar to \cref{lemma-higher-ranks}.
Informally, it states that if we remove a relation from a
presentation for a monoid
$M$, then all of the relations that hold among words representing
elements that are in
strictly higher $\J$-classes of $M$ continue to hold in the
monoid defined by
the presentation with one fewer relation.

\begin{lemma}\label{lemma-assume-higher-ranks}
  Let $M$ be a monoid defined by the presentation $\presn{A}{R}$, let $\phi_A:
  A ^ *\to M$ be the corresponding surjective homomorphism, and let $(u, v)\in
  R$. If $M'$ denotes the monoid defined by the presentation
  $\presn{A}{R\setminus\{(u, v)\}}$ and $\phi_A':A^*\to M'$ is the
  corresponding surjective homomorphism, then $(w)\phi_A' = (w')\phi_A'$ for all
  $w, w' \in A ^*$ such that $(w)\phi_A = (w')\phi_A$,
  $M\cdot (w)\phi_A\cdot M \supsetneq M\cdot
  (u)\phi_A\cdot M$ and $M\cdot (w')\phi_A\cdot M\supsetneq M\cdot
  (u)\phi_A\cdot M$.
\end{lemma}

To illustrate the utility of \cref{lemma-assume-higher-ranks} suppose that $\P$
is any presentation for $T_n$ containing a rank $n - 3$ relation $u = v$. Then
since there is a presentation with no rank $n - 3$ relations, for example any
of those given in \cref{section-intro}, it follows that $u = v$ is redundant in
$\P$ by \cref{lemma-assume-higher-ranks}.

We conclude this section with a simple application of Tietze transformations,
which we will use repeatedly in later sections.

\begin{lemma}\label{lemma-computational-justification}
  Suppose that $\presn{A}{R}$ is a presentation for $S_n$ and that $\presn{A,
  B}{R, S}$ is a presentation for $M$ where $M \in \{I_n, T_n, PT_n\}$. If
  $\presn{C}{T}$ defines $S_n$ and $\zeta_{A, C}:A ^* \to C ^*$ is the
  corresponding change of alphabet mapping, and $S'$ is
  obtained from $S$ by replacing any occurrence of
  $w\in A^*$ in any relation by $(w)\zeta_{A, C}\in C^*$,
  then the presentation $\presn{C, B}{T, S'}$
  also defines $M$.
\end{lemma}

For some small values of $n$, the general arguments given in
\cref{subsection-sym-inv-upper-bound,subsection-Tn-upper-bound,subsection-PTn-upper-bound}
do not apply. In these cases, we will verify various presentations define the
required monoids by computational means. Similar to the presentations given in
the introduction, these presentations will be given in terms of an arbitrary
presentation for the symmetric group.  Obviously, we cannot verify
computationally every such presentation. However,
\cref{lemma-computational-justification} implies that it is sufficient to check
a single presentation for $S_n$.

\section{Generating sets for the alternating group}\label{section-alt-group}

A key step in both of the proofs of the main theorems about the full and
partial transformation monoids
(\cref{theorem-full-transf-main,theorem-PTn-main})
is showing that conjugation by any even permutation fixes a specific generator
(\cref{lemma-conj-by-even,lem:ZetaEvenCommute}). To prove these key steps
we require a number of specific generating sets for the alternating
groups; which we state in this section.

If $f\in S_n$, then we define the \defn{support} of $f$ to be the set
\[
  \supp(f) = \set{i\in \{1, \ldots, n\}}{(i)f \neq i}.
\]
If $f, g\in S_n$, then the \defn{conjugate}  of $f$ by $g$ is $g^{-1}fg$ and is
denoted $f ^ g$. The \defn{commutator} of $f$ and $g$ is $g^{-1}f^{-1}gf$
and is denoted $[f, g]$.

\begin{lemma}[Chapter I, Lemma 6.11
  in~\cite{Hungerford1974aa}]\label{lemma-gen-alt-group-2}
  If $n\in \N$ and $n\geq 3$, then the $3$-cycles of the form $(1, 2, i)$ where
  $i \in \{3, \ldots, n\}$ generate $A_{n}$.
\end{lemma}

\begin{lemma}\label{lemma-gen-alt-group-1}
  If $n\in \N$ and $n\geq 5$, then the $3$-cycles of the form $(1, i, i +2)$
  where $i \in \{2, \ldots, n -2\}$ generate $A_{n}$.
\end{lemma}
\begin{proof}
  We show that the group $G$ generated by the $3$-cycles in the statement
  contains the generating set from \cref{lemma-gen-alt-group-2}.
  If $i\in \{2, \ldots, n\}$ is such that $2i \leq n$, then
  \[
    (1, 2, 2i) = \prod_{k = 1}^{i-1}(1, 2k, 2k + 2)\in G.
  \]
  If $i\in \{1, \ldots, n\}$ is such that $2i + 3 \leq n$, then
  \begin{align*}
    (1, 2, 2i + 1) & = [(1, 2, 4),(1, 2i + 1, 2i + 3) ] \in G\\
    (1, 2, 2i + 3)& = [(1, 2, 4),(1, 2i + 3, 2i + 1)]\in G.
  \end{align*}
  Therefore $(1, 2, i)\in G$ for all $i\in \{3, \ldots, n\}$.
\end{proof}

\begin{lemma}\label{conj:tauconjugates}
  Let $n$ be even and $n \geq 8$. If $\alpha = (3, 4, \ldots, n)$, $\beta = (3,
  7, 6, 4, 5)$ (as in the even case of \cref{theorem-full-transf-main}), and
  $\tau = \beta\alpha\beta^{-1}\alpha\beta^{-1}\alpha^{-1}\beta\alpha^{-1}$,
  then $(\beta\alpha^{-1}) ^ 2$, $\tau$, and $\tau ^ {\alpha\beta^{-1}}$
  generate $A_{\{3,\dots,n\}}$.
\end{lemma}
\begin{proof}
  We set $G = \langle(\beta\alpha^{-1})^2, \tau,
  \tau ^{\alpha\beta^{-1}}\rangle$. Since the generators are even
  permutations, it follows that $G\leq A_{\{3,\dots,n\}}$.

  If $n = 8$, then it is possible to verify that $G = A_{\{3, \ldots,
  8\}}$ directly using \GAP.

  Suppose that $n > 8$. Then
  \begin{align*}
    \beta\alpha^{-1} & = (3,6)(5,n,n -1,\ldots, 8,7)\\
    (\beta\alpha^{-1})^2 & = (5,n -1,n-3,\ldots, 7,n,n -2, \ldots, 8)\\
    \tau & = (3 ,7 ,n)(4 ,5 ,n-1).
  \end{align*}
  It follows by inspection that $G$ is transitive on $\{3, \ldots, n\}$. By
  conjugating $\tau$ and $\tau^{\alpha\beta^{-1}}$ by $(\alpha\beta^{-1})^2$,
  it follows that $\tau^{(\beta\alpha^{-1})^{-k}}\in G$ for all $k\in \N$.
  Since $(\beta\alpha^{-1})^{n-5} = (3, 6)$,
  \[
    \tau\cdot (\tau^{-1})^{(\beta\alpha^{-1})^{n-5}} = (3,7,n)(6,n,7) =
    (3,6,n) \in G.
  \]
  Hence conjugating $(3, 6, n)$ by $(\beta\alpha^{-1})^{-2k}$ for $1 \leq k <
  n-5$, yields $(3, 6, i)\in G$ where $5 \leq i \leq n$ and $i\neq 6$.
  \cref{lemma-gen-alt-group-2} states that these $3$-cycles generate $A_{\{3,
  5, \ldots, n\}}\leq G$. Since $G$ is transitive on $\{3, \ldots, n\}$ and
  contains $A_{\{3, 5, \ldots, n\}}$, it follows that $G = A_{\{3,\dots,n\}}$
  as required.
\end{proof}

\begin{lemma}\label{conj:GenerateAlternatingZeta-odd}
  Let $n$ be odd and $n\geq 7$. If $\alpha = (3, \ldots, n)$, $\beta = (4, 6)$
  (as in the odd case of \cref{theorem-PTn-main}), and $\tau =
  \beta\alpha\beta^{-1}\alpha\beta^{-1}\alpha^{-1}\beta\alpha^{-1}$, then
  $(\beta\alpha^{-1})^2$, $\tau$, and $\tau^{\alpha\beta^{-1}}$ generate
  $A_{\{3,\dots,n\}}$.
\end{lemma}
\begin{proof}
  As in the early proofs in this section, we set $G = \langle
  (\beta\alpha^{-1})^2, \tau, \tau ^ {\alpha\beta^{-1}}\rangle\leq
  A_{\{3,\dots,n\}}$.

  It is straightforward to show that
  \begin{align*}
    \beta\alpha ^ {-1} & = (3, n, n -1, \ldots, 7, 6)(4, 5)\\
    (\beta\alpha ^{-1}) ^ 2 & = (3, n -1, n -3, \ldots, 6, n, n -2,
    \ldots, 7) \\
    \tau &= (4, 6, n).
  \end{align*}
  It follows that conjugating $\tau$ by $(\beta\alpha^{-1})^2$ we obtain every
  $3$-cycle of the form $(4, i, j)\in G$ where $i\in
  \supp((\beta\alpha^{-1})^2) = \{3, \ldots, n\}\setminus \{4, 5\}$
  and $j = (i)(\beta\alpha^{-1})^2$ (i.e.\ $i$ and
    $j$ are consecutive in the $(n - 2)$-cycle
  $(\beta\alpha^{-1})^2$). If $i$ and $j$ are arbitrary, then conjugating $(4,
  i, j)$ by $(4, 6, j)$ yields $(4, 6, i)$. Since $(4, 6, n)\in G$ and $n\in
  \supp((\beta\alpha^{-1})^2)$, it follows that $(4, 6, i)\in G$ for all $i\in
  \supp((\beta\alpha^{-1})^2)\setminus\{4, 6\}$. Therefore by
  \cref{lemma-gen-alt-group-2}, $G$ contains $A_{\{3,\ldots,
  n\}\setminus \{5\}}$.
  Clearly $G$ is transitive on $\{3, \ldots, n\}$, and it contains the
  stabiliser of $5$ in $A_{\{3, \ldots, n\}}$ and so $G = A_{\{3, \ldots,
  n\}}$.
\end{proof}

\begin{lemma}\label{conj:GenerateAlternatingZeta-even}
  Let $n$ be even and $n> 7$. If $\alpha = (3, 5, 4)$, $\beta = (3, \ldots, n)$
  (as in the even case of \cref{theorem-PTn-main}), and $\tau =
  \beta\alpha\beta^{-1}\alpha\beta^{-1}\alpha^{-1}\beta\alpha^{-1}$, then
  $(\beta\alpha^{-1})^2$, $\tau$, and
  $\tau ^{\alpha\beta^{-1}}$ generate $A_{\{3,\dots,n\}}$.
\end{lemma}
\begin{proof}
  As in the early proofs in this section, we set $G = \langle
  (\beta\alpha^{-1})^2, \tau, \tau ^ {\alpha\beta^{-1}}\rangle\leq
  A_{\{3,\dots,n\}}$.

  It is straightforward to show that
  \begin{align*}
    \beta\alpha ^ {-1} & = (3, 5, 6, \ldots, n, 4)\\
    (\beta\alpha^{-1})^2 & = (3, 6, 8, \ldots, n)(4, 5, 7, \ldots, n -1)\\
    \tau &= (3, n, 4, 6, 5) \\
    \tau^{\alpha\beta^{-1}} & = (3, 4, n - 1, n, 5).
  \end{align*}
  Direct computation shows that $(\tau^{\alpha\beta^{-1}})^3 \cdot \tau ^
  2\cdot \tau^{\alpha\beta^{-1}}\cdot \tau = (3, 5,
  4)\in G$ and that \[\tau(\beta\alpha^{-1})^2= (4, 8, 10, \ldots, n, 5, 6, 7,
      9, \ldots, n -
  1)\in G.\] It follows that $(3, 5, 4) ^ {(\beta\alpha^{-1})^2 \tau} = (3, 5,
  7)\in G$. By conjugating $(3, 5, 7)$ by powers of
  $\tau(\beta\alpha^{-1})^2$ we obtain every $3$-cycle of the form $(3, i,
  j)$ where $(i)(\tau(\beta\alpha^{-1})^2)^2 = j$ belongs to $G$. Thus by
  \cref{lemma-gen-alt-group-1}, it follows that $G =A_{\{3, \ldots, n\}}$.
\end{proof}

\begin{lemma}\label{conj:GenerateAlternatingEta-odd}
  Let $n$ be odd and  $n \geq 7$. If $\gamma = (2, \ldots, n)$ and $\delta =
  (2, 3)(4, 6)$ (as in the odd case of \cref{theorem-PTn-main}), and $\rho =
  \delta\gamma\delta^{-1}\gamma\delta^{-1}\gamma^{-1}\delta\gamma^{-1}$, then
  $(\delta\gamma^{-1})^2$, $\rho$, and $\rho^{\gamma\delta^{-1}}$ generate
  $A_{\{2,\dots,n\}}$.
\end{lemma}
\begin{proof}
  We set $G = \langle (\delta\gamma^{-1})^2, \rho, \rho ^
  {\gamma\delta^{-1}}\rangle\leq A_{\{2, \ldots, n\}}$.

  If $n = 7$, then it is possible to verify that $G = A_{\{2, \ldots,
  7\}}$ directly using \GAP.

  Suppose that $n > 7$. It is straightforward to show that
  \begin{align*}
    \delta\gamma ^ {-1} & = (3, n, n -1, \ldots, 6)(4, 5)\\
    (\delta\gamma^{-1})^2 & = (3, n - 1, n - 3, \ldots, 6, n, n-2, \ldots, 7)\\
    \rho &= (2, 3, n -1)(4, 6, n) \\
    \rho^{\gamma\delta^{-1}} & = (2, 6, n)(3, 5, 7).
  \end{align*}
  It is routine to verify that $(\rho ^{(\delta\gamma^{-1})^{n - 5}})  ^{-1} =
  (2, n, 6)(3, 7, 4)\in G$ and so
  \[
    (\rho ^{(\delta\gamma^{-1})^{n - 5}})  ^{-1} \cdot
    \rho^{\gamma\delta^{-1}} =
    (3, 7, 4)(3, 5, 7) = (4, 5, 7) \in G.
  \]
  Thus conjugating $(4, 5, 7)$ by powers of  $(\delta\gamma^{-1})^2$ yields
  $(4, 5, i)\in G$ for all $i\in \supp( (\delta\gamma^{-1})^2) = \{3, \ldots,
  n\}\setminus \{4, 5\}$. Hence by \cref{lemma-gen-alt-group-2}, it follows
  that $A_{\{3, \ldots, n\}} \leq G$. But $G$ is clearly transitive on $\{2, 3,
  \ldots, n\}$ and contains the stabiliser of $2$ in $A_{\{2, \ldots, n\}}$,
  and hence $G = A_{\{2, \ldots, n\}}$.
\end{proof}

\begin{lemma}\label{conj:GenerateAlternatingEta-even}
  Let $n$ be even and  $n > 7$. If $\gamma = (2, 3, 5, 4)$ and $\delta =
  (2, 3, \ldots, n)$ (as in the even case of \cref{theorem-PTn-main}), and
  $\rho =
  \delta\gamma\delta^{-1}\gamma\delta^{-1}\gamma^{-1}\delta\gamma^{-1}$, then
  $(\delta\gamma^{-1})^2$, $\rho$, and $\rho^{\gamma\delta^{-1}}$ generate
  $A_{\{2,\dots,n\}}$.
\end{lemma}
\begin{proof}
  We set $G = \langle (\delta\gamma^{-1})^2, \rho, \rho ^
  {\gamma\delta^{-1}}\rangle\leq A_{\{2, \ldots, n\}}$.

  Routine computation shows that
  \begin{align*}
    \delta\gamma ^ {-1} & = (3, 5, 6, \ldots, n, 4)\\
    (\delta\gamma^{-1})^2 & = (3, 6, 8, \ldots, n)(4, 5, 7, \ldots, n - 1)\\
    \rho &= (2,4,6,5)(3,n) \\
    \rho^{\gamma\delta^{-1}} & = (2, n, 5, 3)(4, n - 1).
  \end{align*}
  Hence $(\rho^2) ^ {(\gamma\delta ^{-1})} = (2, 5)(3, n)$ and so
  \[
    (\rho^2) ^ {(\gamma\delta ^{-1})} \cdot \rho = (2, 5)(2, 4, 6, 5) = (4, 6,
    5)\in G.
  \]
  Conjugating $(4, 6, 5)$ by $\rho ^ {-1}$ yields $(2, 4, 6)\in G$.
  We obtain an $(n-3)$-cycle with support $\{2, \ldots, n\}\setminus \{2, 4\}$
  as follows:
  \[
    (3, 5, 7, \ldots, n - 1, 6, 8, \ldots, n) = (\delta\gamma^{-1})^2\cdot
    (\rho^{\gamma\delta^{-1}})^2\cdot \rho \in G.
  \]
  Thus conjugating $(2, 4, 6)$ by powers of this $(n-3)$-cycle we get $(2, 4,
  i)\in G$ for all $i \in \{2, \ldots, n\}\setminus \{2, 4\}$. Hence
  \cref{lemma-gen-alt-group-2} implies that $G = A_{\{2, \ldots, n\}}$.
\end{proof}

\section{Symmetric inverse monoid}\label{section-In}

In this section we will prove \cref{theorem-In-main}. To prove
\cref{theorem-In-main} we must do two things: show that for all $n\in \N$ with
$n\geq 4$, every presentation for $I_n$ requires at least $3$ relations in
addition to those defining the symmetric group
(\cref{section-In-lower-bound});
and that the presentations given in \cref{theorem-In-main} actually define
$I_n$ (\cref{subsection-sym-inv-upper-bound}). In the final subsection of this
section, \cref{section-In-leq-3}, we consider presentations for $I_n$ for the
small values of $n$ not covered by \cref{theorem-In-main}.

\subsection{A lower bound of $3$}\label{section-In-lower-bound}

Throughout this section we suppose that $n \geq 4$. In this section we will
show that if $\presn{A}{R}$ is any monoid presentation for the symmetric group
$S_n$ and $\mathcal{P} = \presn{ A, B }{ R, S}$ is a monoid presentation for
$I_n$, then $|S| \geq 3$.

We may suppose without loss of generality (by \cref{cor-presentation-subsets})
that $A\cup B$ is an irredundant generating set for $I_n$. We may further
suppose without loss of generality that $B = \{\vep\}$ where $\vep\in I_n$ is
the rank $n - 1$ idempotent with $1 \notin \dom(\vep)$. We denote by $\phi:
(A\cup \{\vep\})^* \to I_n$ the natural surjective homomorphism extending the
inclusion of $A\cup \{\vep\}$ in $I_n$.

It suffices to prove the following theorem.
\begin{theorem}\label{theorem-In-lower-bound}
  If $n\geq 3$ and $\presn{A}{R}$ is any monoid presentation for the symmetric
  group $S_n$ and $\presn{A, \vep}{R, S}$ is a monoid presentation for $I_n$,
  then $|S| \geq 3$.
\end{theorem}

We will show that $S$ must contain at least two rank $n - 1$ relations, and at
least one rank $n - 2$ relation.

\begin{lemma}\label{lemma-leading-perm-In}
  If $u, v\in (A\cup \{\vep\}) ^*$ are such that $(u)\phi = (v)\phi$
  and both $u$
  and $v$ have rank $n - 1$, then the leading permutations of $u$ and $v$
  belong to the same left coset of the stabiliser $\Stab(1) = \set{f\in
  S_n}{(1)f = 1}$ of $1$ in $S_n$.
\end{lemma}
\begin{proof}
  Since $u$ and $v$ have rank $n - 1$, we may write $u = u_1\vep u_2$ and
  $v = v_1\vep v_2$  where $u_2, v_2\in A ^*$ and $(u_1)\phi,
  (v_1)\phi\in S_n$ are the leading permutations of $u$ and $v$ respectively.
  Hence $(\vep u_2)\phi = ((u_1)\phi)^{-1}(v_1\vep v_2)\phi$ is of
  rank $n - 1$
  also. The unique point not in $\dom((\vep u_2)\phi)$ is $1$.
  Hence the only
  point not in $\dom\big(((u_1)\phi)^{-1}\cdot (v_1)\phi \cdot \vep
  \phi\big)$ is also $1$.
  Therefore $((u_1)\phi)^{-1} \cdot (v_1)\phi\in \Stab(1)$, as required.
\end{proof}

If $g\in A ^*$ is arbitrary, then, using the relations in $R$, it is easy to
see that $I_n$ is also defined by the presentation $\presn{A, B}{R,
S\setminus\{(u, v)\}\cup \{(gu, gv)\}}$. Therefore, by
\cref{lemma-leading-perm-In}, we may assume without loss of generality that
every relation $u = v$ in $S$ of rank $n - 1$ has the form $p\vep w = \vep w'$
for some $p\in A^*$ and $w, w'\in (A\cup \{\vep\}) ^ *$, where $p$ is the
leading permutation of $p\vep w$ and $(p)\phi \in \Stab(1)$.

\begin{lemma}\label{lemma-generating-stab-In}
  If the rank $n-1$ relations in $S$ are $p_i\vep w_i= \vep w_i'$ for $i\in
  \{1, \ldots, k\}$ for some $k\in \N$ and where $p_i$ is the leading
  permutation of the left-hand side of each relation, then $\Stab(1)$ is
  generated by $\{(p_1)\phi, \ldots, (p_k)\phi\}$.
\end{lemma}
\begin{proof}
  If $\sigma \in \Stab(1)$ is any non-identity element and $q\in A^*$ is such
  that $(q)\phi = \sigma$, then the equality $q\vep = \vep q$ holds in the
  presentation. In
  particular, there is an elementary sequence $\alpha_1 = \vep q, \alpha_2,
  \ldots, \alpha_m = q\vep$ with respect to the relations $R \cup S$. Since
  this sequence starts at a rank $n - 1$ word, \cref{lemma-assume-higher-ranks}
  tells us that the relations used for the transitions may only be of rank $n$
  (i.e.\ relations from $R$) or rank $n - 1$. Roughly speaking, we will show
  that the identity (the leading permutation of $\vep q$) is transformed by
  this elementary sequence into $q$ (the leading permutation of $q\vep$). We
  will use this to show that $\sigma$ belongs to the subgroup generated by
  $\{(p_1)\phi, \ldots, (p_k)\phi\}$.

  For every $i$, we denote the leading permutation of $\alpha_i$ by $v_i$.
  If a transition $\alpha_i, \alpha_{i + 1}$ in the elementary sequence is via
  $R$, then the permutations represented by $v_i$ and $v_{i + 1}$ are the same
  and so $(v_i) \phi = (v_{i + 1}) \phi$. Since the respective leading
  permutations of $\alpha_1$ and $\alpha_m$ are the identity and $q$, and
  $\sigma$ is a non-identity element of $\Stab(1)$, there exists $i$ such that
  $(v_i)\phi \neq (v_{i + 1})\phi$ and so the transition $\alpha_i, \alpha_{i +
  1}$ is via a rank $n - 1$ relation $p_j \vep w_j = \vep w_j'$ for some $j$.

  We will show that one of the following
  holds:
  \begin{equation*}
    (v_{i + 1}) \phi = (v_i)\phi \cdot
    (p_j)\phi, \quad \text{or }\quad (v_{i + 1}) \phi = (v_i) \phi \cdot
    ((p_j) \phi)^{-1}.
  \end{equation*}
  There are two possibilities for how the words on either side of the rank
  $n - 1$ relation can occur within $\alpha_i$ and $\alpha_{i + 1}$:
  \begin{enumerate}
    \item
      there exists $u\in (A\cup B)^*$ such that
      $\alpha_i = v_i \cdot \vep w_j'\cdot u$ and $\alpha_{i + 1} =
      v_i \cdot p_j \vep w_j\cdot u$ and
      $v_ip_j = v_{i + 1}$ (i.e.\ $v_ip_j$ is the leading permutation of
      $\alpha_{i + 1}$); or
    \item
      there exists $u\in (A\cup B)^*$ such that
      $\alpha_{i + 1} = v_{i + 1} \cdot \vep w_j'\cdot u$ and
      $\alpha_i = v_{i + 1} \cdot p_j \vep w_j\cdot u$ and $v_{i +
      1}p_j = v_i$.
  \end{enumerate}
  If (1) holds, then $(v_{i + 1})\phi = (v_i) \phi \cdot (p_j) \phi$.
  If (2) holds, then $(v_i) \phi = (v_{i + 1}) \phi \cdot (p_j)
  \phi$, whereby $(v_{i + 1}) \phi = (v_i) \phi \cdot ((p_j) \phi)^{-1}$.

  Applying the above argument to every $i$ such that
  $(v_i)\phi \neq (v_{i + 1})\phi$ shows that $\sigma \in
  \langle (p_1)\phi, \ldots, (p_k)\phi \rangle$ and so $\{(p_1) \phi, \ldots,
  (p_k) \phi\}$ generates $\Stab(1)$.
\end{proof}

\begin{corollary}\label{lemma-In-at-least-2-n-1}
  If $n \geq 4$, then there are at least two rank $n - 1$ relations in $S$.
\end{corollary}
\begin{proof}
  Clearly $\Stab(1) \cong S_{n - 1}$. If $n \geq 4$, then $\Stab(1)$ is not
  cyclic, and so $k$ in \cref{lemma-generating-stab-In} is at least $2$.
\end{proof}

It is possible to show that the conclusion of \cref{lemma-In-at-least-2-n-1}
holds when $n = 3$, but the proof given above does not work. We have opted not
to include the proof in the case of $n = 3$ for the sake of brevity. Next we
will show that there is at least one rank $n - 2$ relation in $S$.

\begin{lemma}\label{lemma-In-at-least-1-n-2}
  There is at least one relation of rank $n - 2$ in $S$.
\end{lemma}
\begin{proof}
  Suppose that $u, v\in (A\cup \{\vep\}) ^*$ have rank $n -2$ and
  that $(u)\phi = (v)\phi$, and that the leading permutations of $u$
  and $v$ belong to
  distinct cosets of
  $\Stab(1)$. Then there exists an elementary sequence $\alpha_1 = u,
  \alpha_2, \ldots, \alpha_m = v$ with respect to $R$ and $S$. The defining
  relations $R$ for $S_n$ preserve the leading permutation of any word, and
  the rank $n - 1$ relations in $S$ preserve the coset of the leading
  permutation in $\Stab(1)$ (by \cref{lemma-leading-perm-In}). It follows
  that there exists $k\in \{1, \ldots, m-1\}$ such that the transition from
  $\alpha_k$ to $\alpha_{k + 1}$ uses a rank $n - 2$ relation $u_1 = v_1$ in
  $S$, where the leading permutations of $\alpha_k$ and $\alpha_{k + 1}$
  belong to different cosets of $\Stab(1)$.

  By the preceding paragraph, to show that $S$ contains a rank $n - 2$
  relation, it suffices to show that there exists $f\in I_n$ such
  that $f$ has
  rank $n - 2$ and $f=(u)\phi=(v)\phi$ for some $u, v\in (A\cup \{\vep\})
  ^*$ with leading
  permutations in distinct cosets of $\Stab(1)$. Such $f\in I_n$
  are numerous,
  for example, if
  \[
    f =
    \begin{pmatrix}
      1 & 2 & 3 & 4 & \cdots & n\\
      - & - & 3 & 4 & \cdots & n \\
    \end{pmatrix},
  \]
  then $f$ can be factorised as both $\vep (1, 2) \vep$ and $(1, 2) \vep (1,
  2) \vep$ with leading permutations the identity and $(1, 2)$,
  respectively.
\end{proof}

This concludes the proof of \cref{theorem-In-lower-bound}, since by
\cref{lemma-In-at-least-2-n-1} we know $S$ contains at least two
rank $n - 1$
relations, and by \cref{lemma-In-at-least-1-n-2} it contains at
least one rank
$n - 2$ relation.

\subsection{The proof of
\cref{theorem-In-main}}\label{subsection-sym-inv-upper-bound}

In this section, we prove that the presentation $\mathcal{I}$ given in
\cref{theorem-In-main} defines $I_n$ when $n \geq 3$.

Throughout this subsection, we set
\[\beta = (2, 3, \ldots, n), \quad\text{and}\quad \tau_k = (k, k + 1)\in S_n\]
for appropriate values of $k$.

It is routine to verify that the relations~\ref{rel-I6} and~\ref{rel-I7} hold
in $I_n$. Hence it suffices to show that the
relations~\ref{rel-I1},~\ref{rel-I3},~\ref{rel-I4}, and~\ref{rel-I5} from
\cref{theorem-In-five-rels} are consequences of the relations $R$, and the
relations~\ref{rel-I2},~\ref{rel-I6}, and~\ref{rel-I7}.

\begin{lemma}\label{lemma-In-transpositions}
  If $k \in \{2, \ldots, n - 1\}$ and $\tau_{k - 1}\vep \tau_{k - 1} =
  \vep$ holds in $\mathcal{I}$, then  $\tau_{k}\vep \tau_{k} =
  \vep$ holds in $\mathcal{I}$ also.
\end{lemma}
\begin{proof}
  With some straightforward rearrangement the relation~\ref{rel-I6} becomes
  \begin{align}\label{eq-B1-2}
    \beta \vep \beta^{-1} &= \vep^2.
  \end{align}
  By the assumption in the statement,
  \begin{equation}\label{eq-B1-3}
    \tau_{k-1} \vep^2 \tau_{k-1} = \tau_{k-1} \vep \tau_{k-1} \cdot \tau_{k-1}
    \vep \tau_{k-1}
    = \vep^2.
  \end{equation}

  Combining~\eqref{eq-B1-2} and~\eqref{eq-B1-3} yields $\tau_{k-1} \vep^2
  \tau_{k-1} = \beta \vep \beta^{-1}$, which can be rearranged to
  \begin{equation}\label{eq-B1-4}
    \vep^2 = \tau_{k-1} \beta \vep \beta^{-1} \tau_{k-1}.
  \end{equation}
  This implies
  \[
    \beta \vep \beta^{-1} \stackrel{\eqref{eq-B1-2}}{=} \vep^2
    \stackrel{\eqref{eq-B1-4}}{=} \tau_{k-1} \beta \vep \beta^{-1} \tau_{k-1}.
  \]
  Rearranging the previous equation gives $\left(\beta^{-1} \tau_{k-1}
  \beta\right) \vep \left(\beta^{-1} \tau_{k-1} \beta\right) = \vep$ and
  so, since $\beta^{-1} \tau_{k-1} \beta = \tau_{k}$ holds in $S_n$,
  $\tau_{k} \vep \tau_{k} = \vep$, as required.
\end{proof}

\begin{corollary}\label{lemma-In-conj}
  If $\sigma \in A^*$ represents an element of $\Stab(1)$, then $\sigma^{-1}
  \vep \sigma = \vep$ holds in $\mathcal{I}$.
\end{corollary}
\begin{proof}
  Rearranging the relation~\ref{rel-I2} yields
  $\tau_2 \vep \tau_2 = \vep$.
  Hence, by repeatedly applying \cref{lemma-In-transpositions},
  \begin{equation}\label{eq-B1-6}
    \tau_{k} \vep \tau_{k} = \vep,
  \end{equation}
  holds in $\mathcal{I}$ for all $k\in\{2, 3, \ldots, n - 1\}$. Since
  $\{\tau_{k} \ \mid \ 2 \leq k \leq n - 1\}$ is a generating set for
  $\Stab(1)$, the result follows.
\end{proof}

We can now show that the relations~\ref{rel-I1},~\ref{rel-I3},~\ref{rel-I4},
and~\ref{rel-I5} hold in $\mathcal{I}$:
\begin{enumerate}[label=(B\arabic*)]
  \item[(\ref{rel-I1})]
    \cref{lemma-In-conj} implies $\beta \vep \beta^{-1} = \vep$
    and it follows
    by~\ref{rel-I6} that $\beta \vep \beta^{-1} = \vep^2$. Hence the
    relation $\vep
    ^ 2 = \vep$ holds in $\mathcal{I}$, which is precisely~\ref{rel-I1}.
  \item[(\ref{rel-I3})] $\vep\beta = \beta\vep$ holds in
    $\mathcal{I}$ by~\ref{rel-I6}
    and~\ref{rel-I1}.
  \item[(\ref{rel-I4})]
    $\vep \tau_1 \vep \tau_1 = \vep \tau_1 \vep$ holds in
    $\mathcal{I}$ since
    \begin{align*}
      \vep \tau_1 \vep \cdot \tau_1
      &= \tau_1 \vep \tau_1 \vep \tau_1 \vep \tau_1
      \cdot \tau_1 && \text{by~\ref{rel-I7}}\\
      &=  \tau_1 \vep \tau_1 \cdot \vep \tau_1 \vep\\
      & = \tau_1 \vep \tau_1 \cdot \tau_1
      \vep \tau_1 \vep \tau_1 \vep \tau_1 && \text{by~\ref{rel-I7}}\\
      &= \tau_1 \vep^2 \tau_1 \vep \tau_1 \vep \tau_1\\
      & = \tau_1 \vep \tau_1 \vep \tau_1
      \vep \tau_1 && \text{by~\ref{rel-I1}} \\
      & = \vep \tau_1 \vep && \text{by~\ref{rel-I7}}.
    \end{align*}
  \item[(\ref{rel-I5})]
    $\tau_1 \vep \tau_1\vep = \vep \tau_1 \vep$ holds in
    $\mathcal{I}$ by a
    similar argument to the previous case.
\end{enumerate}

\subsection{The $n = 1$, $2$, $3$ cases}\label{section-In-leq-3}
In this section we consider the ``corner cases'' of presentations for $I_n$
when $n = 1$, $2$ and $3$, since these are not covered by
\cref{theorem-In-main}. \GAP and Python code for performing the computations
mentioned in this section can be found in~\cite{libsemigroupsShortCaseStudy}.

When $n = 1$, it is routine to verify that the monoid presentation
$\presn{x}{x^2=x}$ defines $I_1$ and has the least possible
number of relations.

For $n = 2$, it can be verified using, for example, the
\Semigroups package for
\GAP, that the presentation
\[
  \presn{x, \vep}{x^2=1,\quad \vep^2 = \vep,\quad (x \vep)^3 x =
  \vep x \vep}
\]
defines $I_2$ (where $x$ represents the transposition $(1, 2)$). Hence, by
\cref{lemma-computational-justification}, if $\presn{A}{R}$ is
any presentation
for $S_2$, then the presentation
\[
  \presn{A, \vep}{R, \quad\vep^2 = \vep, \quad((1, 2) \vep)^3 (1, 2)
  = \vep (1, 2) \vep}
\]
defines $I_2$.
It can be shown that any presentation for $I_n$ for any $n\geq 2$ over the
generating set $A\cup \{\vep\}$ contains at least one rank $n - 1$ relation
where the number of occurrences of $\vep$ is not preserved. Similarly, it is
possible to show that any presentation for $I_n$ with $n \geq 2$ must contain a
relation which does not preserve the left coset of $\Stab(1)$ of the leading
permutation. It is possible use the preceding two sentences to show that the
number of non-$S_n$ relations in any presentation of $I_2$ is at least $2$, and
so the presentation given above has the minimum number of non-$S_n$ relations.

\cref{theorem-In-main} provides a presentation for $I_3$ with $3$ non-$S_n$
relations. However, the argument in \cref{section-In-lower-bound} does not
apply to $I_3$. In particular, \cref{lemma-In-at-least-2-n-1} does not hold
when $n = 3$ because $\Stab(1)$ is cyclic (of order $2$).  However, it is
possible to prove, using a different argument to that given in
\cref{section-In-lower-bound}, that any presentation for $I_n$, $n \geq 3$
requires at least two rank $n-1$ relations. This follows since such a
presentation requires at least one relation satisfying:
\begin{itemize}
  \item two occurrences of $\vep$ on one side of the relation are
    replaced by one occurrence on the other; and
  \item the left coset of the leading permutation is not invariant.
\end{itemize}
It is possible to prove that a relation satisfying both of these conditions
cannot be the only rank $n - 1$ relation in a presentation defining $I_n$. In
particular, the minimum number of non-$S_n$ relations in any presentation for
$I_3$ is $3$.

\section{The full transformation monoid}\label{section-Tn}

Throughout this section we will denote the presentation in
\cref{theorem-full-transf-main} by $\T_{odd}$ when $n\geq 5$ is odd and
$\T_{even}$ when $n\geq 8$ is even.

In this section we will prove \cref{theorem-full-transf-main}. We do this in
two steps: we show that for all $n\geq 4$, every presentation for $T_n$
requires at least $4$ non-$S_n$ relations (\cref{subsection-Tn-lower-bound});
and we show that the presentations $\T_{odd}$ and $\T_{even}$ define $T_n$ for
the given values of $n$. For the second step, we first show that $T_n$, $n\geq
4$ is defined by an alternative presentation $\T$ with $5$ non-$S_n$ relations
(\cref{subsection-Tn-upper-bound}). We then show that $\T_{odd}$ and
$\T_{even}$ define $T_n$ by showing that the relations in $\T$ hold
(\cref{subsection-proof-full-transf}). Finally,  we consider the cases not
covered by the presentation in \cref{theorem-full-transf-main} when $n = 1$,
$2$, $3$, $4$, and $6$ (\cref{section-Tn-leq-4}).

\subsection{A lower bound of $4$}\label{subsection-Tn-lower-bound}

Throughout this section we suppose that $n \geq 4$. In this
section we will show
that if $\presn{A}{R}$ is any monoid presentation for the symmetric
group $S_n$ and $\mathcal{P} = \presn{A, B}{R, S}$ is a monoid
presentation for $T_n$, then $|S| \geq 4$.

Similar to~\cref{section-In-lower-bound}, we may suppose without loss
of generality that $B = \{\ve\}$ where $\ve\in T_n$ is the
idempotent with rank $n - 1$ and $(2)\ve = 1$. We denote by $\phi: (A\cup
\{\ve\})^* \to T_n$ the natural surjective homomorphism extending the inclusion
of $A\cup \{\ve\}$ in $T_n$.

It suffices to prove the following theorem.
\begin{theorem}\label{theorem-Tn-lower-bound}
  If $n\geq 4$, $\presn{A}{R}$ is any monoid presentation for the symmetric
  group $S_n$, and $\presn{A, \ve}{R, S}$ is a monoid presentation for $T_n$,
  then $|S| \geq 4$.
\end{theorem}

We will show that $S$ must contain at least two rank $n - 1$ relations, and at
least two rank $n - 2$ relations.

The following lemma is analogous to \cref{lemma-leading-perm-In}.

\begin{lemma}\label{lemma-leading-perm}
  If $u, v\in (A\cup \{\ve\}) ^*$ are such that $(u)\phi = (v)\phi$ in
  $T_n$ and $u$ and $v$
  have rank $n - 1$, then the leading permutations of $u$ and $v$
  belong to the
  same left coset of the setwise stabiliser $\Stab(\{1, 2\}) =
  \set{f\in S_n}{\{1,
  2\}f = \{1, 2\}}$ of $\{1, 2\}$ in $S_n$.
\end{lemma}
\begin{proof}
  Since $u$ and $v$ have rank $n - 1$, we may write $u = u_1\ve u_2$ and
  $v = v_1\ve v_2$  where $u_2, v_2\in (A\cup \{\ve\})^*$ and $u_1, v_1\in A^*$
  are the leading permutations of $u$ and $v$ respectively.  Hence $(\ve
  u_2)\phi = ((u_1)\phi)^{-1}(v_1\ve v_2)\phi$ is of rank $n - 1$ also. The only
  non-trivial class of the kernel of $(\ve u_2)\phi$ is $\{1, 2\}$. Hence the
  only non-trivial class of $((u_1)\phi)^{-1}(v_1\ve)\phi$ is $\{1, 2\}$ also.
  Therefore $((u_1)\phi)^{-1} \cdot (v_1)\phi\in \Stab(\{1, 2\})$, as required.
\end{proof}

By a similar argument to that given in \cref{section-In-lower-bound}, we may
assume that every relation in $S$ of rank $n - 1$ has the form $p\ve u = \ve v$
for some $p\in A ^*$ and $u, v\in (A\cup \{\ve\}) ^ *$ where $(p)\phi \in
\Stab(\{1, 2\})$.

The next lemma is analogous to \cref{lemma-generating-stab-In} and the proof is
omitted for the sake of brevity.

\begin{lemma}\label{lemma-generating-stab}
  If the rank $n-1$ relations in $S$ are $p_i\ve u_i= \ve v_i$
  for $i\in \{1,
  \ldots, k\}$ for some $k\in \N$ where $(p_i)\phi\in \Stab(\{1, 2\})$, then
  $\Stab(\{1, 2\})$ is generated by $\{(p_1)\phi, \ldots, (p_k)\phi\}$.
\end{lemma}

An immediate corollary of \cref{lemma-generating-stab} tells us the
minimum number of rank $n - 1$ relations in any presentation for
$T_n$ when $n \geq 4$.

\begin{corollary}\label{lemma-Tn-at-least-2-n-1}
  There are at least two rank $n -1$ relations in $S$.
\end{corollary}
\begin{proof}
  Clearly, $\Stab(\{1, 2\}) \cong S_{2} \times S_{n - 2}$. Hence, since $n \geq
  4$, $\Stab(\{1, 2\})$ is not cyclic, and hence $k$ in
  \cref{lemma-generating-stab} is at least $2$.
\end{proof}

To show that there must be at least two rank $n - 2$
relations in any presentation for $T_n$ we require the following notion.
If $f\in T_n$, then we say that $f$ has \defn{kernel type} $a_1^{b_1}\cdots
a_k^{b_k}$ where $a_1, \ldots, a_k, b_1, \ldots, b_k\in \{1, \ldots, n\}$ and
$a_1b_1+ \cdots + a_k b_k = n$ to denote that $f$ has $b_i$ kernel
classes of size $a_i$ for every $i$.

We also require the following lemma.

\begin{lemma}\label{lemma-new-1}
  Suppose that $f\in T_n$ has rank $n- 2$ and is such that $f =
  (w_1)\phi = (w_2)\phi$ for some
  $w_1, w_2\in (A\cup\{\ve\})^*$ for which the leading permutations
  of $w_1$ and $w_2$
  belong to distinct cosets of $\Stab(\{1, 2\})$. Then there exists a rank $n
  -2$ relation $u=v\in S$ such that $(u)\phi$ (and $(v)\phi$) have the same
  kernel type as $f$.
\end{lemma}
\begin{proof}
  Suppose that $f\in T_n$ is such that $f = (w_1)\phi = (w_2)\phi$ for some
  $w_1, w_2\in (A\cup\{\ve\})^*$ for which the leading permutations
  of $w_1$ and $w_2$
  belong to distinct cosets of $\Stab(\{1, 2\})$. Then there exists an
  elementary sequence $\alpha_1 = w_1, \alpha_2, \ldots, \alpha_k = w_2$ with
  respect to $R$ and $S$. Since the leading permutations of $w_1$ and $w_2$
  belong to distinct cosets of $\Stab(\{1, 2\})$, there exists $i$ such that
  the leading permutations of $\alpha_i$ and $\alpha_{i + 1}$ belong to
  distinct cosets of $\Stab(\{1, 2\})$. In particular,
  there exist $p, q\in (A\cup \{\ve\})^*$ such that $\alpha_i = puq$ and
  $\alpha_{i + 1} = pvq$ for some $u=v\in R\cup S$. If $u=v\in R$, the
  defining relations for $S_n$, then, the
  leading permutations of $\alpha_i$ and $\alpha_{i+1}$ would be identical.
  Hence $u = v\in S$.
  If $p\not\in A ^*$, then $(p)\phi\notin S_n$ and so the leading permutations
  of $\alpha_i = puq$ and $\alpha_{i + 1} = pvq$ coincide with
  the leading permutation of $p$,  contradicting the
  assumption that the leading permutations of $\alpha_i$ and $\alpha_{i+1}$
  belong to distinct cosets of $\Stab(\{1, 2\})$. Hence
  $p\in A^*$ and $(p)\phi$ is a permutation.

  By \cref{lemma-higher-ranks}, the rank of $u=v$ is not less than $n - 2$. If
  the rank of $u=v$ is $n$, then $(u)\phi = (v)\phi$ is a
  permutation and so the leading permutations of $\alpha_i = p uq$
  and $\alpha_{i + 1}= p vq$ would coincide, which is also a contradiction. If
  the rank of $u=v$ is $n - 1$, then, by \cref{lemma-leading-perm}, the
  leading permutations of $pu$ and $pv$, and hence $puq$ and $pvq$, belong to
  the same coset of $\Stab(\{1, 2\})$, which is a contradiction. Hence $u=v$
  has rank $n - 2$.

  Since $\ker(gh) \supseteq \ker(g)$ for all $g,h\in T_n$, if the ranks of $gh$
  and $g$ coincide, then $\ker(gh) = \ker(g)$. We have shown that $(u)\phi$ has
  rank $n - 2$, and by assumption $(uq)\phi$ has rank $n - 2$ also. Thus
  $\ker((uq)\phi) = \ker((u)\phi)$. Since $(p)\phi$ is a permutation,
  $(puq)\phi$ and $(uq)\phi$ have the same kernel type, and so $(\alpha_i)\phi$
  and $(u)\phi$ have the same kernel type. Therefore the kernel types of
  $(u)\phi$ and $(\alpha_j)\phi=f$ coincide for all $j$.
\end{proof}

\begin{lemma}\label{lemma-Tn-at-least-2-n-2}
  There are at least two rank $n - 2$ relations in $S$.
\end{lemma}
\begin{proof}
  The possible kernel types of rank $n - 2$ transformations in $T_n$ are
  $3^{1}1^{n-3}$ and $2^21^{n -4}$. Hence to show that there are at least two
  rank $n- 2$ relations, it suffices by \cref{lemma-new-1} to show that for
  each kernel type $3^{1}1^{n-3}$ and $2^21^{n -4}$ there exists a
  transformation $f$ with that kernel type, and words
  $w_1, w_2\in (A\cup\{\ve\})^*$ such
  that $f = (w_1)\phi = (w_2)\phi$ where the leading permutations of $w_1$ and
  $w_2$ belong to distinct cosets of $\Stab(\{1, 2\})$.

  \begin{enumerate}[label=\bf Case \arabic*., leftmargin=0pt, labelwidth=!,
      itemsep=1em, ref=Case \arabic*]
    \item
      Suppose that $w_1 = \ve (2, 3) \ve$, $w_2 = (2, 3) \ve (2, 3) \ve$, and
      \[
        f = (w_1)\phi = (w_2)\phi =
        \begin{pmatrix}
          1 & 2 & 3 & 4 & \cdots & n\\
          1 & 1 & 1 & 4 & \cdots & n \\
        \end{pmatrix}.
      \]
      Then $f$ has kernel type $3^11^{n-3}$ and the leading permutations of
      $w_1$ and $w_2$ are the identity and $(2, 3)$ respectively. Hence the
      leading permutations belong to distinct cosets of $\Stab(\{1, 2\})$, as
      required.

    \item
      Suppose that $w_1 = (1, 3) (2, 4)\ve (1, 3)(2, 4) \ve$,
      $w_2 = \ve (1, 3) (2, 4)\ve (1, 3)(2, 4)$, and
      \[
        f =
        \begin{pmatrix}
          1 & 2 & 3 & 4 & 5 & \cdots & n\\
          1 & 1 & 3 & 3 & 5 & \cdots & n \\
        \end{pmatrix}.
      \]
      Then $f = (w_1)\phi = (w_2)\phi$ has kernel type $2^21^{n - 4}$
      and the leading permutations of
      $w_1$ and $w_2$ belong to distinct cosets of $\Stab(\{1, 2\})$.\qedhere
  \end{enumerate}
\end{proof}

This concludes the proof of \cref{theorem-Tn-lower-bound} since
the set $S$ must contain at least two relations of rank $n - 1$ by
\cref{lemma-Tn-at-least-2-n-1} and at least two relations of rank $n - 2$
by \cref{lemma-Tn-at-least-2-n-2}.

\subsection{A presentation with $5$ non-$S_n$
relations}\label{subsection-Tn-upper-bound}

The purpose of this section is to prove the following theorem.
\begin{theorem}\label{thm-full-transf-5-extra-rels}
  Suppose that $n\in \N$, that $n\geq 5$, and $\presn{A}{R}$ is any
  monoid presentation for the symmetric group $S_n$ of degree $n$.
  Then the following presentation $\mathcal{T}$ defines $T_n$:
  the generators are $A\cup \{\ve\}$, and the relations are $R$
  together with:
  \begin{multicols}{2}
    \begin{enumerate}[label=\rm (T\arabic*), ref=\rm T\arabic*]
        \addtocounter{enumi}{8}
      \item [\rm (\ref{rel-T1})]
        $\ve  (1, 3) \ve  (1, 3) = \ve$;
      \item [\rm (\ref{rel-T3})]
        $(3, 4) \ve = \ve  (3, 4)$;
      \item [\rm (\ref{rel-T7})]
        $(1, n)(2, 3) \ve  (1, n)(2, 3) \ve = \ve (1, n)(2, 3)\ve
        (1, n)(2, 3)$;
      \item[\rm (\ref{rel-T8})] $(2, 3)  \ve  (2, 3)  \ve(2, 3)  \ve(2, 3) =
        \ve  (2, 3)   \ve$.
      \item\label{rel-T9}
        $(3, 4, \ldots, n) (1, 2)
        \ve = \ve  (3, 4, \ldots, n)$;
      \item[\vspace{\fill}]
    \end{enumerate}
  \end{multicols}
  \noindent using the notation of \cref{theorem-full-transf-main}.
\end{theorem}
\begin{proof}
  By \cref{theorem-aizenstat}, the relations~\ref{rel-T1},~\ref{rel-T3},
  and~\ref{rel-T7} hold in $T_n$. It is routine to verify that the
  relations~\ref{rel-T8} and~\ref{rel-T9} also hold in $T_n$.
  Thus it suffices
  to show that the remaining
  relations~\ref{rel-T2},~\ref{rel-T4},~\ref{rel-T6},
  and~\ref{rel-T5} in Aizenstat's presentation
  (\cref{theorem-aizenstat}) are consequences of the relations given in
  \cref{thm-full-transf-5-extra-rels}.

  We set $\theta = (3, 4,\ldots, n) \in S_n$.

  \begin{enumerate}[label=\bf (A\arabic*)]
    \item[(\ref{rel-T2})]\label{A3_E1__A2}
      We must show that $(1, 2) \ve = \ve$ holds in
      $\mathcal{T}$.
      If $k \geq 1$, then, by repeated application of~\ref{rel-T9},
      \[
        \theta ^k(1, 2)^k\ve = \ve \theta ^ k
      \]
      for all $k\geq 1$.
      If $n$ is odd, then we set $k = n - 2$. In this case $\theta ^k = \id$
      and $\theta ^ k (1, 2)^ k = (1, 2)$ and so $(1, 2) \ve =
      \ve$ holds in $\mathcal{T}$.

      If $\psi = \theta (3, 4) = (4, 5, \ldots, n)$, then
      \[
        (1, 2) \psi \ve = (1, 2)\theta(3, 4) \ve
        \stackrel{\ref{rel-T3}}{=} (1, 2)\theta \ve (3, 4)
        \stackrel{\ref{rel-T9}}{=}  \ve \theta (3, 4)
        = \ve \psi.
      \]
      Repeatedly applying the equality $(1, 2)\psi\ve = \ve\psi$, we obtain
      \[
        \psi^k(1, 2)^k \ve = \ve \psi ^k
      \]
      for all $k \geq 1$.
      If $n$ is even, then we set $k = n - 3$. In this case $\psi
      ^ {k} = \id$
      and $\psi^k(1, 2)^k = (1, 2)$ and so $(1, 2)\ve = \ve$ holds in
      $\mathcal{T}$.

      Note that we have just shown that the symmetric group relations
      $R$,~\ref{rel-T3}, and~\ref{rel-T9} imply that~\ref{rel-T2}
      holds; this will
      be useful in \cref{subsection-PTn-upper-bound}.

    \item[(\ref{rel-T4})]
      $\ve (3, 4, \ldots, n) = (3, 4, \ldots, n) \ve$ holds in
      $\mathcal{T}$ by~\ref{rel-T2} and~\ref{rel-T9}.

      Again note that~\ref{rel-T2} and~\ref{rel-T9} imply that~\ref{rel-T4}
      holds, which we will use in \cref{subsection-PTn-upper-bound}.
  \end{enumerate}

  At this point we have shown that the relations~\ref{rel-T1}
  to~\ref{rel-T4} hold in $\mathcal{T}$. Since these are the only
  rank $n - 1$
  relations in the presentation in \cref{theorem-aizenstat}, it
  follows that $\ve ^ 2 = \ve$ holds in
  $\mathcal{T}$ by \cref{lemma-higher-ranks}.

  \begin{enumerate}[label=\bf (A\arabic*)]
      \addtocounter{enumi}{4}
    \item[(\ref{rel-T6})]
      $\ve (2, 3) \ve (2, 3) = \ve (2, 3) \ve$ holds in $\mathcal{T}$ since
      \begin{align*}
        \ve (2, 3) \ve \cdot (2, 3) & =
        (2, 3) \ve  (2, 3) \ve  (2, 3) \ve (2, 3)(2, 3) &&
        \text{by~\ref{rel-T8}}\\
        &=  (2, 3)\ve  (2, 3) \cdot \ve  (2, 3) \ve\\
        & = (2, 3)\ve (2, 3) \cdot (2, 3) \ve
        (2, 3) \ve  (2, 3) \ve (2, 3) &&\text{by~\ref{rel-T8}}\\
        &= (2, 3)\ve \ve  (2, 3) \ve  (2, 3) \ve (2, 3)  \\
        &= (2, 3)\ve  (2, 3) \ve  (2, 3) \ve (2, 3)
        && \text{by \cref{lemma-higher-ranks}} \\
        &= \ve (2, 3) \ve &&\text{by~\ref{rel-T8}}.
      \end{align*}
    \item[(\ref{rel-T5})]
      That $(2, 3) \ve (2, 3)\ve = \ve (2, 3) \ve$ holds in
      $\mathcal{T}$ holds
      follows by a symmetric argument to that given above
      for~\ref{rel-T6}.\qedhere
  \end{enumerate}
\end{proof}

It is possible to show that the relations in the presentation in
\cref{thm-full-transf-5-extra-rels} are irredundant, but this is not
particularly relevant here, and so we omit the details.

\subsection{The proof of
\cref{theorem-full-transf-main}}\label{subsection-proof-full-transf}

In this section we show that the presentations $\mathcal{T}_{odd}$ and
$\mathcal{T}_{even}$ given in
\cref{theorem-full-transf-main} define $T_n$ when $n \geq 7$.

We do this, in both cases, by showing that the 5
relations~\ref{rel-T1},~\ref{rel-T3},~\ref{rel-T7},~\ref{rel-T8},~\ref{rel-T9}
in the
presentation $\T$
from \cref{thm-full-transf-5-extra-rels} hold. Since~\ref{rel-T7}
and~\ref{rel-T8} are relations in all three of the presentations $\T_{odd}$,
$\T_{even}$ and $\T$, it suffices to show
that~\ref{rel-T1},~\ref{rel-T3}, and~\ref{rel-T9} are consequences of
the relations in $\T_{odd}$ and $\T_{even}$.

The key step in the proof of \cref{theorem-full-transf-main} is
the following
lemma.

\begin{lemma}\label{lemma-conj-by-even}
  If $\sigma \in A ^*$ represents an (even) permutation in
  $A_{\{3, 4, \ldots, n\}}$, then $\sigma^{-1}\ve \sigma = \ve$ holds in
  $\T_{odd}$ and $\T_{even}$.
\end{lemma}

We will find the following straightforward consequence of the relations in
both $\T_{odd}$ and $\T_{even}$ useful:
\begin{equation}\label{eq-A1-1}
  \beta^{-1} \ve \beta \stackrel{\ref{rel-T-beta}}{=} (1, 2) \cdot
  \ve (1, 3) \ve (1, 3)
  \stackrel{\ref{rel-T-alpha}}{=} (1, 2)\alpha^{-1} \ve \alpha.
\end{equation}

\begin{proof}[Proof of Lemma~\ref{lemma-conj-by-even} when $n$ is odd.]
  Recall from \cref{theorem-full-transf-main} that, since $n$ is odd,
  $\alpha = (3, 4)$ and $\beta = (3, 4, \ldots, n)\in S_n$. We set $\sigma =
  (4, 5, \ldots, n)$. We will prove by induction that
  \[
    (3, j, (j)\sigma ^2)\ve (3, (j)\sigma ^2, j) = \ve
  \]
  for all $j\in \{4, \ldots, n\}$. It will follow from this that
  $(3, k+2, k)
  \ve (3, k, k + 2) = \ve$ for all $k\in  \{4, \ldots, n\}$ and
  the proof will
  be complete by \cref{lemma-gen-alt-group-1}.

  The base case of our induction is when $j = n - 1$. Observe that
  \begin{equation}\label{eq-A1-2}
    \begin{aligned}
      \ve &= (3, 4) \ve (1, 3) \ve (1, 3) (3, 4)&&
      \text{by rearranging~\ref{rel-T-alpha}}\\
      &= (3, 4) \ve (3, 4) \cdot (1, 4)\cdot  (3, 4) \ve
      (3, 4) \cdot (1, 4) && \text{conjugating by }(3, 4)\\
      &= (1, 2)\beta^{-1} \ve \beta
      \cdot (1, 4) \cdot (1, 2) \beta^{-1} \ve \beta \cdot (1, 4)
      &&\text{by~\eqref{eq-A1-1}}.
    \end{aligned}
  \end{equation}
  Conjugating~\eqref{eq-A1-2} by $\beta^{-1}$, we obtain
  \begin{equation}\label{eq-A1-3}
    \beta \ve \beta^{-1} = (1, 2) \ve (1, 3) (1, 2) \ve (1, 3).
  \end{equation}
  Conjugating~\eqref{eq-A1-3} by $\beta^{-1}$, we get
  \begin{equation}\label{eq-1}
    \beta^2\ve \beta^{-2} = (1, 2) \beta \ve \beta^{-1} (1, n)
    (1, 2) \beta \ve \beta^{-1} (1, n).
  \end{equation}
  To proceed, we find a different expression for $\beta \ve
  \beta^{-1}$.  In this direction,
  \[
    (1, 2) \ve  \stackrel{\eqref{eq-A1-1}}{=} \beta (3, 4)\cdot
    \ve\cdot (3, 4) \beta^{-1}
    =  (3, n) \beta\cdot \ve \cdot \beta^{-1} (3, n) \quad
    \text{since }\beta (3, 4) = (3, n)\beta.
  \]
  Hence rearranging the previous equation we obtain
  \begin{equation}\label{eq-A1-4}
    \beta \ve \beta^{-1} =
    (1, 2) (3, n) \ve(3, n).
  \end{equation}
  Substituting~\eqref{eq-A1-4} into~\eqref{eq-1}, we get

  \begin{equation}\label{eq-xxx}
    \begin{aligned}
      \beta^2\ve \beta^{-2} &= (1, 2) \cdot (1
      \ 2) (3, n) \ve (3, n) \cdot (1, n) (1, 2) \cdot (1, 2) (3
      \ n) \ve (3, n) \cdot (1, n)\\
      &= (3, n) \ve (3, n) (1, n) (3, n) \ve (3
      \ n) (1, n)\\
      &= (3, n) \ve (1, 3) \ve (1, 3) (3, n)&& \text{since }
      (3, n)(1, n)=(1, 3)(3, n)\\
      &= (3, n) (3, 4) \ve (3, 4) (3, n) && \text{by~\ref{rel-T-alpha}}.
    \end{aligned}
  \end{equation}
  We will use~\eqref{eq-A1-4} to find a second expression for $\beta^2
  \ve \beta^{-2}$. Conjugating both sides of~\eqref{eq-A1-4} by
  $\beta^{-1}$
  \begin{equation}\label{eq-A1-5}
    \begin{aligned}
      \beta^2\ve \beta^{-2} & = (1, 2) (n - 1, n) \beta
      \ve \beta^{-1} (n - 1, n)\\
      & = (1, 2) (n - 1, n) (1, 2) (3, n) \ve (3, n) (n - 1, n)
      &&\text{by~\eqref{eq-A1-4}}\\
      & =  (n - 1, n) (3, n) \ve (3, n) (n - 1, n).
    \end{aligned}
  \end{equation}
  Equating the right hand sides of~\eqref{eq-xxx} and~\eqref{eq-A1-5} yields
  \(
    (3, n) (3, 4) \ve (3, 4) (3, n) = (n - 1, n) (3, n) \ve (3,
    n) (n - 1, n)
  \),
  and this can be rearranged to
  \begin{equation*}
    \ve = (3, 4, n - 1)\ve (3, n - 1, 4).
  \end{equation*}
  This establishes the base case of our induction.

  Assume that $(3,j, (j)\sigma^2)\ve (3, (j)\sigma^2, j)=\ve$
  for some $j\in \{4, \ldots, n\}$. We show that
  \[(3,(j)\sigma, (j)\sigma^3)\ \ve\ (3, (j)\sigma^3, (j)\sigma)=\ve\]
  as follows:
  \begin{align*}
    (1, 2) (3, 4) \ve (3, 4)
    &= \beta^{-1} \ve \beta && \text{by~\eqref{eq-A1-1}}\\
    & = \beta^{-1}\  (3,j, (j)\sigma^2)\ \ve\ (3, (j)\sigma^2, j)\ \beta&&
    \text{by assumption}\\
    &= (4, (j)\beta, (j)\sigma^2\beta) \beta^{-1} \ve \beta (4,
    (j)\beta, (j)\sigma^2\beta)&&
    \text{as }(3, (j)\sigma^2, j)\ \beta=\beta (4, (j)\beta,
    (j)\sigma^2\beta)\\
    &= (4, (j)\beta, (j)\sigma^2\beta) \cdot (1, 2) (3 \ 4) \ve (3,
    4) \cdot (4, (j)\beta, (j)\sigma^2\beta)
    &&\text{by~\eqref{eq-A1-1}}.
  \end{align*}
  Rearranging the previous equation yields:
  \begin{equation*}
    \ve
    = (4, (j)\beta, (j)\sigma^2\beta)^{(3, 4)} \cdot  \ve
    \cdot (4, (j)\beta, (j)\sigma^2\beta)^{(3, 4)}.
  \end{equation*}
  There are three cases to consider.
  \begin{enumerate}
    \item
      If
      $(j)\beta\neq 3$ and $(j)\sigma^2\beta\neq 3$, it follows that
      \[
        \ve = (3, (j)\beta, (j)\sigma^2\beta) \cdot \ve \cdot(3, (j)\beta,
        (j)\sigma^2\beta).
      \]
      Since $(j)\beta, (j)\sigma^2\beta\in \{5, \ldots, n\}$, it
      follows that
      $j, (j)\sigma^2 \in \{4, \ldots, n - 1\}$. Hence $(j)\beta =
      (j)\sigma$ and
      $(j)\sigma^2\beta = (j)\sigma^3$ and so
      \[
        \ve = (3, (j)\sigma, (j)\sigma^3) \cdot \ve \cdot(3, (j)\sigma,
        (j)\sigma^3),
      \]
      as required.

    \item
      If $(j)\beta =3$, then $j = n$ and $(j)\sigma^2 = 5$, and so
      \begin{align*}
        \ve
        & = (4, (j)\beta, (j)\sigma^2\beta)^{(3, 4)} \cdot  \ve
        \cdot (4, (j)\beta, (j)\sigma^2\beta)^{(3, 4)}\\
        &= (4, 3, 6)^{(3, 4)} \cdot \ve \cdot (4, 6, 3)^{(3, 4)}\\
        &= (3, 4, 6)\cdot \ve \cdot (3, 6, 4)\\
        &= (3, (j)\sigma, (j)\sigma^3)\cdot\ve \cdot (3, (j)\sigma ^3,
        (j)\sigma),
      \end{align*}
      as required.

    \item
      The case when $(j)\sigma^2\beta = 3$ is similar to the previous
      case.\qedhere
  \end{enumerate}
\end{proof}

\begin{proof}[Proof of Lemma~\ref{lemma-conj-by-even} when $n$ is even.]
  Throughout this section we suppose that $n \geq 8$ is even and
  we will write
  $\alpha = (3,\ldots, n)$ and $\beta = (3, 7, 6, 4, 5)$.
  By $(\ref{rel-T-alpha})$ and $(\ref{rel-T-beta})$,
  \begin{align}
    \alpha^{-1}\zeta\alpha &=
    (1,2)\beta^{-1}\zeta\beta,\label{alphabeta2} \\
    \zeta &= (1,2)\beta\alpha^{-1}\zeta\alpha\beta^{-1}.\label{UsefulEq2}
  \end{align}
  Now, from (\ref{rel-T-beta}), we obtain
  \begin{align*}
    \zeta &=(1,2)\beta\zeta(1,3)\zeta(1,3)\beta^{-1} \\
    &=\beta(\beta\alpha^{-1}\zeta
    \alpha\beta^{-1})(1,3)(1,2)(\beta\alpha^{-1}\zeta
    \alpha\beta^{-1})(1,3)\beta^{-1} &\text{by~\eqref{UsefulEq2}.}
  \end{align*}
  Thus, by conjugating by $\beta^2\alpha^{-1}$,
  \begin{align*}
    \alpha\beta^{-2}\zeta\beta^2\alpha^{-1} &=
    \zeta \alpha\beta^{-1}(1,3)(1,2)\beta\alpha^{-1}\zeta \alpha
    \beta^{-1}(1,3)\beta\alpha^{-1} \\
    &=(1,2)(\beta\alpha^{-1}\zeta\alpha\beta^{-1})\alpha\beta^{-1}(1,3)\beta\alpha^{-1}(\beta\alpha^{-1}\zeta\alpha\beta^{-1})\alpha\beta^{-1}(1,3)\beta\alpha^{-1}
    &&\text{by~\eqref{UsefulEq2}}\\
    &=(1,2)\beta\alpha^{-1}\zeta(1,3)\zeta(1,3)\alpha\beta^{-1} \\
    &=\beta\alpha^{-1}(\beta^{-1}\zeta\beta)\alpha\beta^{-1}
    &&\text{by (\ref{rel-T-beta})}
  \end{align*}
  where the third equality holds because $3$ is fixed by
  $\alpha\beta^{-1}\alpha\beta^{-1}$ and hence $(1,3)$ commutes with
  $\alpha\beta^{-1}\alpha\beta^{-1}$. Moreover,
  \begin{align*}
    \alpha\beta^{-2}\zeta\beta^2\alpha^{-1}
    &=(1,2)\alpha\beta^{-1}\alpha^{-1}\zeta\alpha\beta\alpha^{-1}
    &\text{by~\eqref{alphabeta2}} \\
    &=
    \alpha\beta^{-1}\alpha^{-1}(\beta\alpha^{-1}\zeta\alpha\beta^{-1})\alpha\beta\alpha^{-1}
    &\text{by~\eqref{UsefulEq2}.}
  \end{align*}
  So, by equating both expressions we have that
  \begin{equation}\label{tauequal2}
    \zeta =
    \beta\alpha\beta^{-1}\alpha\beta^{-1}\alpha^{-1}\beta\alpha^{-1}\zeta\alpha\beta^{-1}\alpha\beta\alpha^{-1}\beta\alpha^{-1}\beta^{-1}
    = \tau\zeta\tau^{-1}.
  \end{equation}
  Furthermore, we can see that
  \begin{equation}\label{eq-6}
    \begin{aligned}
      \zeta &= (1,2)\beta\alpha^{-1}\zeta\alpha\beta^{-1}
      &&\text{by~\eqref{UsefulEq2}} \\
      &= (1,2)\beta\alpha^{-1}\tau\zeta\tau^{-1}\alpha\beta^{-1}
      &&\text{by~\eqref{tauequal2}} \\
      &=
      (1,2)\beta\alpha^{-1}\tau\alpha(\alpha^{-1}\zeta\alpha)\alpha^{-1}\tau^{-1}\alpha\beta^{-1}
      \\
      &=\beta\alpha^{-1}\tau\alpha\beta^{-1}\zeta\beta
      \alpha^{-1}\tau^{-1}\alpha\beta^{-1} &&\text{by
      (\ref{rel-T-alpha}) and (\ref{rel-T-beta}).}
    \end{aligned}
  \end{equation}
  Moreover, by applying~\eqref{UsefulEq2} twice, we obtain that
  \begin{equation}\label{eq-7}
    \zeta =
    \beta\alpha^{-1}\beta\alpha^{-1}\zeta\alpha\beta^{-1}\alpha\beta^{-1}.
  \end{equation}

  Therefore, by~\eqref{tauequal2},~\eqref{eq-6},~\eqref{eq-7}, and
  \cref{conj:tauconjugates}, $\zeta =
  \sigma^{-1}\zeta\sigma$ for all $\sigma \in A_{\{3,\dots,n\}}$.
\end{proof}

\begin{corollary}\label{eq-sym-conj}
  If $\sigma \in A ^*$ represents an odd permutation in $S_{\{3,
  4, \ldots, n\}}$, then $ \sigma^{-1} \ve \sigma = (1, 2) \ve $ holds in
  $\T_{odd}$ and $\T_{even}$.
\end{corollary}
\begin{proof}
  Suppose that $\sigma \in S_{\{3, 4, \ldots, n\}}$  is odd.
  If $n$ is odd, then $\alpha = (3, 4)$ is an odd permutation, and
  $\beta= (3, 4, \ldots, n)$ is an even permutation.
  If $n$ is even, then $\alpha = (3, 4, \ldots, n)$ is odd and
  $\beta= (3, 7, 6, 4, 5)$ is even. In both cases, $\alpha$ is
  odd and $\beta$
  is even.

  \cref{lemma-conj-by-even} implies that $\beta^{-1} \ve \beta
  = \ve$, and
  so~\eqref{eq-A1-1} may be rearranged to give
  \begin{equation}\label{eq-A1-7}
    \alpha^{-1} \ve \alpha = (1, 2) \ve.
  \end{equation}
  Since $\sigma$ is odd,  $\sigma = \alpha\tau$ for some $\tau\in
  A_{\{3, 4, \ldots, n\}}$.
  Hence
  \begin{align*}
    \sigma^{-1} \ve \sigma &=  \tau
    ^{-1}\alpha^{-1}\ve \alpha\tau
    = \tau^{-1} (1, 2) \ve\tau && \text{by~\eqref{eq-A1-7}}\\
    & = (1, 2)  \tau^{-1} \ve\tau   = (1, 2) \ve &&
    \text{by \cref{lemma-conj-by-even}}.\qedhere
  \end{align*}
\end{proof}

Next, we establish that $\ve$ is an idempotent, which we require later.

\begin{lemma}\label{lemma-idempotent}
  The relation $\ve ^ 2 = \ve$ holds in $\T_{odd}$ and $\T_{even}$.
\end{lemma}

\begin{proof}
  Combining~\eqref{eq-A1-7} and the defining
  relation~\ref{rel-T-alpha} gives
  \begin{equation}\label{eq-A1-8}
    (1, 2) \ve = \ve (1, 3) \ve (1, 3).
  \end{equation}
  If $\tau \in S_{\{3, 4, \ldots, n\}}$ is odd and $(3) \tau = 3$,
  then conjugating the left-hand side of~\eqref{eq-A1-8} by $\tau$ gives
  \begin{equation}\label{eq-A1-9}
    \tau^{-1}  (1, 2) \ve  \tau = (1, 2)  \tau^{-1} \ve
    \tau\stackrel{\eqref{eq-sym-conj}}{=} (1, 2)  (1, 2) \ve = \ve.
  \end{equation}
  Conjugating the right-hand side of~\eqref{eq-A1-8} by $\tau$ yields
  \begin{equation}\label{eq-A1-10}
    \begin{aligned}
      \tau^{-1} \ve (1, 3) \ve (1, 3) \tau
      &= \tau^{-1} \ve \tau \cdot (1, 3) \cdot \tau^{-1} \ve \tau
      \cdot (1, 3)
      && \text{via conjugation, since $(3) \tau = 3$}\\
      &= (1, 2) \ve \cdot (1, 3) \cdot (1, 2) \ve \cdot (1, 3)
      && \text{by \cref{eq-sym-conj}}\\
      &=(1, 2) \ve (1, 3) \cdot \ve (1, 3) \ve (1, 3) \cdot (1, 3)
      && \text{by~\eqref{eq-A1-8}}\\
      &= (1, 2) \cdot \ve (1, 3) \ve (1, 3) \cdot \ve\\
      &=(1, 2) \cdot (1, 2) \ve \cdot \ve
      && \text{by~\eqref{eq-A1-8}}\\
      &= \ve^2.
    \end{aligned}
  \end{equation}
  Equating the right-hand sides of~\eqref{eq-A1-9} and~\eqref{eq-A1-10} tells us
  that $\ve^2 = \ve$.
\end{proof}

\begin{lemma}\label{lemma-assuming-T1}
  If~\ref{rel-T1} holds in $\T_{odd}$ and $\T_{even}$, then the
  relation $\ve = (1, 2)\ve$ also holds.
\end{lemma}
\begin{proof}
  Given that~\ref{rel-T1} holds in $\T_{odd}$ and $\T_{even}$, it
  follows that
  \begin{equation*}
    \begin{aligned}
      (1, 2) \ve \stackrel{\ref{rel-T1}}{=}
      (1, 2)\ve(1, 3) \ve (1, 3)
      \stackrel{\ref{rel-T-beta}}{=}
      \beta^{-1} \ve \beta
    \end{aligned}
  \end{equation*}
  and $\beta^{-1}\ve \beta = \ve$ by
  \cref{lemma-conj-by-even}, since $\beta\in A_{\{3, \ldots, n\}}$.
\end{proof}

We can now show that the relations~\ref{rel-T1},~\ref{rel-T3},
and~\ref{rel-T9}
hold in $\T_{odd}$ and $\T_{even}$:
\begin{enumerate}
  \item[(\ref{rel-T1})] $\ve (1, 3) \ve (1, 3)= \ve$ holds in
    $\T_{odd}$ and $\T_{even}$ since
    \begin{equation*}
      \begin{aligned}
        \ve &= (1, 2) \ve \cdot (1, 3) \ve (1, 3) &&
        \text{rearranging~\eqref{eq-A1-8}}\\
        &=\ve (1, 3) \ve (1, 3) \cdot
        (1, 3) \ve (1, 3) && \text{by~\eqref{eq-A1-8}} \\
        &= \ve (1, 3) \ve^2 (1, 3)\\
        &= \ve (1, 3) \ve (1, 3) && \text{since $\ve^2 = \ve$ by
        \cref{lemma-idempotent}}.
      \end{aligned}
    \end{equation*}

  \item[(\ref{rel-T3})]
    \cref{eq-sym-conj} implies that  $(3,4)\zeta(3,4) = (1,2)\zeta$
    since $(3, 4)$
    is an odd permutation.
    Because we have shown that~\ref{rel-T1} holds in $\T_{odd}$ and
    $\T_{even}$, \cref{lemma-assuming-T1} then implies that
    $(3,4)\zeta(3,4) = (1,2)\zeta = \zeta$.

  \item[(\ref{rel-T9})]
    If $n$ is odd, then~\ref{rel-T9} is $\beta (1, 2)\ve = \ve \beta$.
    By \cref{eq-sym-conj}, since $\beta$ is odd, $\beta (1, 2)\ve
    \beta^{-1} = \ve$ and so
    $\beta (1, 2)\ve = \ve \beta$.
    Hence~\ref{rel-T9} holds when $n$ is odd.

    If $n$ is even,~\ref{rel-T9} is $\alpha (1, 2)\ve = \ve \alpha$.
    Since $\alpha$ is odd, $\alpha \ve = (1, 2)\ve\alpha$ by
    \cref{eq-sym-conj}. Hence
    \[
      \alpha (1, 2) \ve = (1, 2)\alpha\ve = (1, 2)(1, 2)\ve\alpha
      = \ve\alpha.
    \]
\end{enumerate}

\subsection{The $n = 1$ to $6$ cases}\label{section-Tn-leq-4}

Analogous to \cref{section-In-leq-3}, in this section we consider
presentations
for $T_n$ when $n = 1$ to $6$ since these are not covered by
\cref{theorem-full-transf-main}. Note that we discussed the case of $n = 5$
immediately after the statement of the theorem and so refer the
reader to that.
\GAP and Python code for performing the computations
mentioned in this section can be found in~\cite{libsemigroupsShortCaseStudy}.

Clearly when $n = 1$, the monoid $T_1$ is trivial and hence has a
presentation
of length $0$.

For $n = 2$, it can be verified using, for example, the \Semigroups
package for
\GAP, that the presentation
\[
  \presn{a_2, \ve}{a_2^2=1,\quad a_2\ve = \ve,\quad \ve^2 =\ve}
\]
defines $T_2$ (where $a_2$ represents the transposition $(1, 2)$).
Similar to the argument outlined in \cref{section-In-leq-3}, it can be shown
that any presentation for $T_2$ must contain at least two
relations in addition
to those defining the symmetric group.

For $n = 3$, it can be verified computationally that the presentation
\[
  \presn{a_2,a_3,\ve}{a_2^2=1,\quad  a_3 ^ 2 = 1,\quad
    a_2a_3a_2=a_3a_2a_3,\quad  a_2\ve =
    \ve,\quad  \ve a_3\ve a_3= \ve,\quad  (a_2a_3a_2\ve)^3 a_2a_3a_2 =
  \ve a_2a_3a_2 \ve}
\]
defines $T_3$ (where $a_2$ and $a_3$ represent the transpositions $(1, 2)$ and
$(1, 3)$, respectively). It can be
shown that any presentation for $T_3$ contains at least three non-$S_n$
relations, so this presentation also has the minimum possible number.

For $n = 4$, it is possible to verify computationally that the presentation:
\begin{equation}\label{align-full-transf-4}
  \begin{aligned}
    \langle a_2, a_3, a_4, \ve \mid &\ a_2 ^ 2 = a_3 ^ 2 = a_4 ^ 2
    = (a_2a_3)^3 = (a_3a_4) ^ 3 = (a_4a_2)^ 3 = 1,\\
    &\ (a_2a_3a_2a_4)^2 = (a_3a_4a_3a_2)^2 = (a_4a_2a_4a_3) ^ 2 = 1, \\
    &\ a_2\ve = \ve^2 a_3 \ve  a_3,\quad
    (a_4a_2a_3a_2\ve)^2 = (\ve a_4a_2a_3a_2) ^ 2,\\
    &\ a_3a_4a_3 \ve a_3a_4a_3 = \ve a_3 \ve a_3,\quad
    (a_2a_3a_2  \ve)^3 a_2a_3a_2 = \ve  a_2a_3a_2\ve
    \rangle
  \end{aligned}
\end{equation}
defines $T_4$. As discussed in \cref{section-intro},
\cref{theorem-full-transf-main} states that every presentation for $T_4$
contains at least $4$ non-$S_4$ relations, so this presentation also
has the minimum possible number.

For $n = 6$, again it is possible to verify computationally that the following
presentation defines $T_6$:
\begin{equation}\label{align-full-transf-6}
  \begin{aligned}
    \langle a_2, a_3, a_4, a_5, a_6, \ve \mid &\ a_i ^ 2 = 1\ (i =
    2, \ldots 6),\
    (a_ia_{i + 1})^3 = 1\ (i = 2, \ldots 5),\  (a_6a_2)^ 3 = 1,\\
    & \ {(a_{i}a_{i+1}a_{i}a_{j})}^2=1\quad (2\leq i, j\leq 6,\
    j\not\in \{i, i + 1\},\ a_{7} = a_2) \\
    & \
    (\ve a_6a_2a_3a_2)^2 = (a_6a_2a_3a_2 \ve)^2,\\
    &\ (a_2a_3a_2  \ve)^3 a_2a_3a_2 = \ve a_2a_3a_2\ve,\\
    &\ a_3a_4a_3 \ve a_3a_4a_3 = \ve a_3 \ve a_3,\\
    & \ a_3a_6a_5a_4a_3 \ve  a_3a_4a_5a_6a_3 = a_2 \ve
    \rangle.
  \end{aligned}
\end{equation}
By \cref{theorem-full-transf-main} every presentation for $T_6$ contains at
least $4$ non-$S_6$ relations, and the presentation above meets this bound.

%
%
%
%

\section{Partial transformation monoid}\label{section-PTn}

In this section we prove \cref{theorem-PTn-main}.
This section is organised analogously to~\cref{section-Tn}.

\subsection{A lower bound of $8$}\label{subsection-PTn-lower-bound}

Throughout this section we suppose that $n \geq 4$. In this
section we will show
that if $\presn{A}{R}$ is any monoid presentation for the symmetric
group $S_n$ and $\mathcal{P} = \presn{A, B}{R, S}$ is a monoid
presentation for $PT_n$, then $|S| \geq 8$.

As in the previous sections we may suppose without loss of
generality that $B =
\{\ve, \vep\}$ where $\ve \in T_n$ is the idempotent with rank $n - 1$ and
$(2)\ve = 1$, and $\vep \in I_n$ is the idempotent with rank $n -
1$ such that
$\dom(\vep) = \{2, \ldots, n\}$. We denote by $\phi: (A\cup
\{\vep, \ve\})^* \to
PT_n$ the natural surjective homomorphism extending the inclusion of $A\cup
\{\ve, \vep\}$ in $PT_n$.

It suffices to prove the following theorem.
\begin{theorem} \label{thm-atleast8-for-pt}
  If $\presn{A}{R}$ is any monoid presentation for the symmetric group
  $S_n$ and $\presn{A, \ve, \vep}{R, S}$ is a monoid presentation for
  $PT_n$ for $n\geq 4$, then $|S| \geq 8$.
\end{theorem}

As in the previous sections, we consider the rank $n - 1$ and $n - 2$ cases
separately. We begin by making the following straightforward observation about
leading non-permutations in rank $n - 1$.

\begin{lemma}\label{lemma-leading-epsilon-determines-kind}
  Suppose that $u \in (A\cup \{\vep,\ve\})^*$ has rank $n - 1$. Then
  $u$ has leading
  non-permutation $\ve$ if and only if $(u) \phi \in T_n$.
  Similarly, $u$ has
  leading non-permutation $\vep$ if and only if $(u) \phi \in I_n$.
\end{lemma}

\begin{proof}
  Suppose that $u \in (A\cup \{\vep,\ve\})^*$ has rank $n - 1$. We
  write $u = \sigma \varepsilon
  w$, where $\sigma$ is the leading permutation of $u$ and $\varepsilon\in
  \{\ve, \vep\}$ is the leading non-permutation. Clearly such a leading
  non-permutation exists because $u$ does not have rank $n$.

  If $\varepsilon = \ve$, then the prefix $\sigma \ve$ of $u$ represents a
  rank $n - 1$ full transformation in $PT_n$. Since $u$ itself has rank $n -
  1$, it follows that $\ker(\sigma\ve) = \ker(u)$, so $u\in T_n$.

  If $\varepsilon = \vep$, then the prefix $\sigma \vep$ of $u$
  represents an injective partial transformation of rank $n - 1$.
  Since $u$ also has rank $n - 1$,  it follows that $\ker(\sigma\vep)
  = \ker(u)$, and hence $u$ must also be an
  injective partial transformation, as required.
\end{proof}

An immediate corollary of
\cref{lemma-leading-epsilon-determines-kind} is the
following.

\begin{corollary}\label{lemma-leading-epsilon-invariant-rank-n-1}
  If $u, v \in (A\cup \{\vep,\ve\})^*$ are of rank $n - 1$ and $(u)
  \phi = (v) \phi$, then the
  leading non-permutations of $u$ and $v$ are the same.\qed
\end{corollary}

Since rank $n - 1$ words have leading non-permutation either
$\ve$ or $\vep$,
and both $T_n$ and $I_n$ embed in $PT_n$, an argument similar to
that used in
the proof of \cref{lemma-leading-perm-In,lemma-leading-perm} can be
used to prove
the following.

\begin{lemma}[cf.
  \cref{lemma-leading-perm-In,lemma-leading-perm}]\label{lemma-PTn-leading-perm}
  Suppose $u, v\in (A\cup \{\vep,\ve\}) ^*$ are such that $(u)\phi =
  (v)\phi$, and $u$ and $v$
  have rank $n - 1$. Then the following hold:
  \begin{enumerate}[\rm (a)]
    \item if the leading non-permutation of $u$ and $v$ is $\vep$,
      then the leading
      permutations of $u$ and $v$ belong to the same left coset
      of $\Stab(1)$.

    \item
      if the leading non-permutation of $u$ and $v$ is $\ve$,
      then the leading
      permutations of $u$ and $v$ belong to the same left coset
      of $\Stab(\{1,
      2\})$;
      \qedhere
  \end{enumerate}
\end{lemma}

Again via similar arguments, we can see also that there must be rank $n - 1$
relations concerning words whose leading non-permutation is $\vep$ as in
\cref{lemma-generating-stab-In}, and concerning those whose leading
non-permutation is $\ve$ as in \cref{lemma-generating-stab}.
We make this
precise in the following lemma.

\begin{lemma}[cf.
  \cref{lemma-generating-stab-In,lemma-generating-stab}]\label{lemma-generating-stab-PTn}
  Suppose that $\lambda\in \{\vep, \ve\}$ and the rank $n-1$ relations in $S$
  with leading non-permutation $\lambda$ are $q_i\lambda v_i= \lambda v_i'$ for
  $i\in \{1, \ldots, l\}$. Then the following hold:
  \begin{enumerate}[\rm (a)]
    \item
      $\Stab(1)$ is generated by $\{(q_1)\phi, \ldots, (q_l)\phi\}$
      when $\lambda = \vep$;
    \item
      $\Stab(\{1, 2\})$ is generated by $\{(q_1)\phi, \ldots,
      (q_l)\phi\}$ when $\lambda = \ve$.
  \end{enumerate}
\end{lemma}

By \cref{lemma-leading-epsilon-invariant-rank-n-1}, the leading non-permutation
of a rank $n - 1$ word is uniquely determined, and so these respective families
of relations must be distinct. Since neither $\Stab(\{1, 2\})$ nor $\Stab(1)$
is cyclic, we obtain the following.

\begin{corollary}\label{lemma-at-least-4-n-1-PTn}
  There are at least four rank $n - 1$ relations in $S$. \qed
\end{corollary}

Next, we will show that there must be at least $4$ relations of rank $n - 2$ in
any presentation for $PT_n$.  In order to do this we extend the notion of the
kernel type of an element of $T_n$ to elements of $PT_n$. If $f\in PT_n$,
then we will say that $f$ has \defn{kernel type} $a_1^{b_1}\cdots a_k^{b_k}$
where $a_1, \ldots, a_k, b_1, \ldots, b_k\in \{1, \ldots, n\}$ and $a_1b_1+
\cdots + a_k b_k = |\dom(f)|$ to denote that $f$ has $b_i$ kernel classes of
size $a_i$ for every $i$.

We require the following analogue of~\cref{lemma-new-1} in the context of
$PT_n$.

\begin{lemma}\label{lemma-new-2}
  Suppose that $f\in PT_n$ has rank $n - 2$ and is such that $f =
  (w_1)\phi = (w_2)\phi$ for some
  $w_1, w_2\in (A\cup\{\ve,\vep\})^*$ and one of the following holds:
  \begin{enumerate}[\rm (a)]
    \item the leading non-permutation of $w_1$ and $w_2$ is $\vep$ and
      the leading permutations of $w_1$ and $w_2$
      belong to distinct cosets of $\Stab(1)$;
    \item the leading non-permutation of $w_1$ and $w_2$ is $\ve$ and
      the leading permutations of $w_1$ and $w_2$
      belong to distinct cosets of $\Stab(\{1, 2\})$;
    \item
      the leading non-permutations of $w_1$ and $w_2$ are distinct.
  \end{enumerate}
  Then there exists a rank $n -2$ relation $u=v\in S$ such that $(u)\phi$ (and
  $(v)\phi$) have the same kernel type as $f$.
\end{lemma}

The proof of \cref{lemma-new-2} is similar to that of \cref{lemma-new-1}, and
is omitted for the sake of brevity.

\begin{lemma}\label{lemma-at-least-4-n-2-PTn}
  There are at least four rank $n - 2$ relations in $S$.
\end{lemma}
\begin{proof}
  The possible kernel types of rank $n - 2$ partial transformations in $PT_n$
  are $1^{n-2}$, $3^{1}1^{n-3}$, $2^21^{n -4}$, and $2^11^{n-3}$.
  Hence by \cref{lemma-new-2} it
  suffices for each of these kernel types to show that there exists $f\in PT_n$
  and $w_1, w_2\in (A\cup\{\ve,\vep\})^*$ satisfying one of the conditions in
  \cref{lemma-new-2}. We consider each of these kernel types in the cases
  below.

  \begin{enumerate}[label=\bf Case \arabic*., leftmargin=0pt, labelwidth=!,
      itemsep=1em, ref=Case \arabic*]

    \item
      Suppose that
      $w_1=\vep (1, 2) \vep (1, 2)$,
      $w_2 = (1, 2) \vep (1, 2) \vep$, and
      \[
        f=(w_1)\phi=(w_2)\phi=
        \begin{pmatrix}
          1 & 2 & 3 & 4 & 5 & \cdots & n\\
          - & - & 3 & 4 & 5 & \cdots & n \\
        \end{pmatrix}.
      \]
      The leading non-permutation of $w_1$ and $w_2$ is $\vep$, but the cosets
      of the leading permutations $\id$ and $(1, 2)$ in $\Stab(1)$ are
      distinct. Thus by \cref{lemma-new-2}(a) there is a rank $n - 2$ relation
      $u_1 = v_1$ in $S$, where $(u_1)\phi$ and $(v_1)\phi$ have the same
      kernel type $1^{n - 2}$ as $f$.

    \item Suppose that
      $w_1 = \ve (2, 3) \ve (2, 3)$,
      $w_2= (2, 3) \ve (2, 3) \ve$, and
      \[
        f = (w_1)\phi=(w_2)\phi=
        \begin{pmatrix}
          1 & 2 & 3 & 4 & 5 & \cdots & n\\
          1 & 1 & 1 & 4 & 5 & \cdots & n \\
        \end{pmatrix}.
      \]
      The words $w_1$ and $w_2$ have the same leading non-permutation $\ve$,
      but their leading permutations $\id$ and $(2, 3)$ belong to different
      cosets of $\Stab(\{1, 2\})$.
      Hence by \cref{lemma-new-2}(b) there is a rank $n - 2$ relation
      $u_2 = v_2$ in $S$, where $(u_2)\phi$ and $(v_2)\phi$ have the same
      kernel type $3^1 1^{n - 3}$ as $f$.

    \item Suppose that
      $w_1=(1, 3) (2, 4)\ve (1, 3)(2, 4) \ve$, $w_2=\ve (1,
      3) (2, 4)\ve (1, 3)(2, 4)$, and
      \[
        f = (w_1)\phi = (w_2)\phi=
        \begin{pmatrix}
          1 & 2 & 3 & 4 & 5 & \cdots & n\\
          1 & 1 & 3 & 3 & 5 & \cdots & n \\
        \end{pmatrix}.
      \]
      The words $w_1$ and $w_2$ have the same leading non-permutation $\ve$ but
      their leading permutations $(1, 3)(2, 4)$ and $\id$ belong to distinct
      cosets of $\Stab(\{1, 2\})$. Therefore by \cref{lemma-new-2}(b) there is
      a rank $n - 2$ relation $u_3 = v_3$ in $S$, where $(u_3)\phi$ and
      $(v_3)\phi$ have the same kernel type $2^2 1^{n - 4}$ as $f$.

    \item
      Suppose that $w_1 = \ve (1, 3) \vep (1, 3)$, $w_2 = (1, 3) \vep (1,
      3) \ve$, and
      \[
        f = (w_1)\phi = (w_2)\phi =
        \begin{pmatrix}
          1 & 2 & 3 & 4 & 5 & \cdots & n\\
          1 & 1 & - & 4 & 5 & \cdots & n \\
        \end{pmatrix}.
      \]
      The leading non-permutations of $w_1$ and $w_2$ are distinct. Therefore,
      by \cref{lemma-new-2}(c), there is a rank $n - 2$ relation $u_4 = v_4$ in
      $S$, where $(u_4)\phi$ and $(v_4)\phi$ have the same kernel type $2^11^{n
      - 3}$ as $f$.\qedhere
  \end{enumerate}
\end{proof}

\subsection{A presentation with $9$ non-$S_n$
relations}\label{subsection-PTn-upper-bound}

In this section we prove the following theorem.

\begin{theorem}\label{theorem-PTn-9-extra}
  Suppose that $n\in \N$, that $n \geq 4$, and that $\presn{A}{R}$ is
  any monoid presentation for the symmetric group $S_n$ of degree
  $n$. Then the following presentation $\PT'$ defines $PT_n$:
  the generators are $A\cup \{\ve, \vep\}$, and
  the relations are $R$ together with:
  \setlength{\columnsep}{2cm}
  \begin{multicols}{2}
    \begin{enumerate}[label=\rm (P\arabic*), ref=\rm P\arabic*]
        \addtocounter{enumi}{6}
      \item [\rm (\ref{rel-I3})]
        $(2, 3, \ldots, n)   \vep = \vep  (2,  3, \ldots, n)$;
      \item [\rm (\ref{rel-T3})] $(3, 4)  \ve = \ve  (3, 4)$;
      \item [\rm (\ref{rel-T7})]
        $(1, n)(2, 3) \ve  (1, n)(2, 3)  \ve = \ve (1, n)(2, 3)  \ve
        (1, n)(2, 3)$;
      \item [\rm (\ref{rel-T8})]
        $(2, 3)  \ve  (2, 3)  \ve(2, 3)  \ve(2, 3) = \ve  (2, 3)   \ve$;
      \item [\rm (\ref{rel-T9})]
        $(3,  4,\ldots, n) (1,  2)
        \ve = \ve  (3,  4,  \ldots, n)$;
      \item [\rm (\ref{rel-P1})] $\ve  (1,  2)   \vep  (1,  2)  = \ve$;
      \item [\rm (\ref{rel-P5})]
        $(1,  3)   \vep  (1,  3)   \ve = \ve  (1,  3)   \vep  (1,  3)$;
      \item[\rm (\ref{rel-P6})]
        $(1,  2)   \vep  (1,  2)   \vep  (1,  2) = \ve \vep$;
      \item\label{rel-P7}
        $(2,  3)   \vep  (2,  3)  = \vep  (1,  2)   \ve  (1,  2) $;
      \item[\vspace{\fill}]
    \end{enumerate}
  \end{multicols}
  \noindent using the notation of \cref{theorem-PTn-main}.
\end{theorem}

It is possible to show that the relations in the presentation in
\cref{theorem-PTn-9-extra} are irredundant, but this is not
particularly relevant here, and so we omit the details.

We begin by showing that it suffices to
prove that any relations defining $I_n$ and $T_n$ hold in $\PT'$.
\begin{lemma}\label{lemma-PT-suffices}
  Suppose that $\presn{A}{R}$ is any monoid presentation for $S_n$, and that
  $\P_I = \presn{A, \vep}{R, R_I}$ and $\P_T = \presn{A, \ve}{R, R_T}$ are
  presentations defining $I_n$ and $T_n$, respectively, where
  $\ve$ and $\vep$
  are as given in \cref{theorem-PTn-main}. If the relations in
  $\P_I$ and $\P_T$ hold in $\PT'$, then $\PT'$ defines the
  partial transformation monoid $PT_n$.
\end{lemma}

\begin{proof}
  To show that $\PT'$ defines $PT_n$ it suffices to show that the
  relations in
  $\PT'$ hold in $PT_n$; and the relations in \cref{theorem-PTn-east} hold in
  $\PT'$. It is routine to verify the former.

  The presentation $\PT'$ and the presentation given in
  \cref{theorem-PTn-east} share the
  relations~\ref{rel-I3},~\ref{rel-T3},~\ref{rel-T7},~\ref{rel-P1},
  and~\ref{rel-P5}.  Thus it suffices to prove that the remaining relations
  from the presentation given in \cref{theorem-PTn-east} hold in
  $\PT'$, i.e.\
  that~\ref{rel-I2},~\ref{rel-I4},~\ref{rel-T2},~\ref{rel-T4},~\ref{rel-P2},~\ref{rel-P3}
  and~\ref{rel-P4} hold in $\PT'$.

  The relations~\ref{rel-I2} and~\ref{rel-I4} hold in $I_n$ and the
  relations~\ref{rel-T2},~\ref{rel-T4}, and~\ref{rel-P4} hold in
  $T_n$. Hence these relations hold in $\P_I$ and $\P_T$,
  respectively, and so
  they also hold in $\PT'$ by the assumption of this lemma.

  That~\ref{rel-P2} ($\vep(1, 2)\ve(1, 2) = \vep$) holds
  in $\PT'$ follows from:
  \[
    \vep (1, 2) \ve (1, 2) \stackrel{\ref{rel-P7}}{=} (2, 3) \vep
    (2, 3) \stackrel{\ref{rel-I2}}{=} \vep.
  \]

  It is routine to verify that $(1, 2) \vep (1, 2) \vep (1, 2) = \vep (1,
  2)\vep$ holds in $I_n$ and so holds in $\P_I$. Thus~\ref{rel-P3}
  ($\ve \vep =
  \vep (1, 2) \vep)$ holds in $\PT'$ since
  \[
    \ve \vep \stackrel{\ref{rel-P6}}{=} (1, 2) \vep (1, 2) \vep (1,
    2) = \vep (1, 2) \vep.\qedhere
  \]
\end{proof}

\begin{corollary}\label{cor-PT-suffices-2}
  If the relations~\ref{rel-T1},~\ref{rel-I2},~\ref{rel-I6},
  and~\ref{rel-I7}
  hold in $\PT'$, then $\PT'$ defines the partial
  transformation monoid $PT_n$.
\end{corollary}
\begin{proof}
  By \cref{lemma-PT-suffices}, it suffices to show that the relations in the
  presentations $\mathcal{I}$ from \cref{theorem-In-main} and $\mathcal{T}$
  from \cref{thm-full-transf-5-extra-rels} hold in $\PT'$. More
  specifically, that the
  relations~\ref{rel-I2},~\ref{rel-I6},~\ref{rel-I7},~\ref{rel-T1},~\ref{rel-T3},~\ref{rel-T7},~\ref{rel-T8},
  and~\ref{rel-T9} hold
  in $\PT'$. The relations~\ref{rel-T3},~\ref{rel-T7},~\ref{rel-T8},
  and~\ref{rel-T9} belong to $\PT'$, and the
  relations~\ref{rel-T1},~\ref{rel-I2},~\ref{rel-I6},
  and~\ref{rel-I7} hold by
  assumption.
\end{proof}

By \cref{cor-PT-suffices-2}, to show that $\PT'$ defines $PT_n$ it
suffices to show that the relations~\ref{rel-T1},~\ref{rel-I2},~\ref{rel-I6},
and~\ref{rel-I7} hold in $\PT'$. We will make use of the following
lemma repeatedly when showing this.

\begin{lemma}\label{lemma-B1-holds-in-PT}
  The relation~\ref{rel-I1}, $\vep ^ 2 = \vep$ holds in $\PT'$.
\end{lemma}
\begin{proof}
  We begin by showing that
  \begin{equation}\label{eq-C1-1}
    (2, 3) \vep (2, 3)  \vep \stackrel{\ref{rel-P7}}{=} \vep (1,
    2) \ve (1, 2) \vep
    \stackrel{\ref{rel-P1}}{=} \vep (1, 2)  \ve (1, 2)
    \stackrel{\ref{rel-P7}}{=} (2, 3) \vep (2, 3).
  \end{equation}
  Rearranging the previous equation yields $\vep = \vep (2, 3) \vep (2,
  3)$, and so
  \begin{equation*}
    \vep^2 = \vep (2, 3) \vep (2, 3) \cdot \vep
    = \vep \cdot (2, 3) \vep (2, 3) \vep
    \stackrel{\eqref{eq-C1-1}}{=} \vep \cdot (2, 3) \vep (2, 3)
    = \vep. \qedhere
  \end{equation*}
\end{proof}

We are now in a position to verify that the
relations~\ref{rel-I2},~\ref{rel-T1},~\ref{rel-I6}, and~\ref{rel-I7} hold in
$\PT'$:
\begin{enumerate}
  \item [(\ref{rel-I2})]
    $\vep (2, 3) = (2, 3)\vep$ holds in $\PT'$ by rearranging
    \begin{align*}
      \vep \stackrel{\eqref{eq-C1-1}}{=} \vep \cdot (2, 3) \vep (2, 3)
      \stackrel{\ref{rel-P7}}{=} \vep \cdot \vep (1, 2) \ve (1, 2)
      \stackrel{\ref{rel-I1}}{=} \vep (1, 2) \ve (1, 2)
      \stackrel{\ref{rel-P7}}{=} (2, 3) \vep (2, 3).
    \end{align*}
  \item [(\ref{rel-T1})] We will show that~\ref{rel-T1} holds in
    $\PT'$
    by showing that the rank $n - 1$ relations in the
    presentation for $PT_n$
    in \cref{theorem-PTn-east} hold in $\PT'$. It will
    follow from this
    that every rank $n-1$ relation that holds in any presentation for $PT_n$
    holds in $\PT'$. In particular, any rank $n-1$ relation, such
    as~\ref{rel-T1}, in any presentation for $T_n$ holds in $\PT'$.

    The rank $n - 1$ relations in \cref{theorem-PTn-east}
    are~\ref{rel-T2},~\ref{rel-T3},~\ref{rel-T4},~\ref{rel-I2},~\ref{rel-I3},~\ref{rel-P1},
    and~\ref{rel-P2}. Of these~\ref{rel-T3},~\ref{rel-I3},
    and~\ref{rel-P1} are among the
    defining relations in $\PT'$ and we have already shown
    that~\ref{rel-I2} holds in $\PT'$. Hence it suffices to show
    that~\ref{rel-T2},~\ref{rel-T4}, and~\ref{rel-P2} hold in
    $\PT'$:
    \begin{enumerate}
      \item[(\ref{rel-T2})] We showed in the proof of
        \cref{thm-full-transf-5-extra-rels} in
        \cref{subsection-Tn-upper-bound}
        that the relations for the symmetric group,~\ref{rel-T3},
        and~\ref{rel-T9} imply that~\ref{rel-T2} holds.
        Since~\ref{rel-T3} and~\ref{rel-T9} belong to
        $\PT'$,~\ref{rel-T2} holds in $\PT'$.
      \item[(\ref{rel-T4})] We showed in
        \cref{subsection-Tn-upper-bound}
        that, together with the relations for the symmetric
        group,~\ref{rel-T2} and~\ref{rel-T9} imply that~\ref{rel-T4}.
        We have just shown that~\ref{rel-T2} holds in $\PT'$
        and~\ref{rel-T9} belongs to $\PT'$, and so~\ref{rel-T4} holds
        in $\PT'$.
      \item[(\ref{rel-P2})]
        $\vep = \vep (1, 2) \ve (1, 2)$ holds in $\PT'$ since
        \(
          \vep \stackrel{\ref{rel-I2}}{=} (2, 3) \vep (2, 3)
          \stackrel{\ref{rel-P7}}{=} \vep (1, 2) \ve (1, 2).
        \)
    \end{enumerate}

  \item [(\ref{rel-I6})] $(2, 3, \ldots, n)\vep = \vep^2(2, 3,
    \ldots, n)$ holds in
    $\PT'$ since
    \[
      (2, 3, \ldots, n) \vep \stackrel{\ref{rel-I3}}{=} \vep(2,
      3, \ldots, n)
      \stackrel{\ref{rel-I1}}{=} \vep ^ 2(2, 3, \ldots, n).
    \]
  \item [(\ref{rel-I7})] $(1, 2) \vep (1, 2) \vep (1, 2) \cdot
    \vep (1, 2) =\vep (1, 2) \vep$ holds in $\PT'$ since
    \begin{align*}
      (1, 2) \vep (1, 2) \vep (1, 2) \cdot \vep (1, 2)
      &= \ve \vep \cdot \vep (1, 2)
      && \text{by~\ref{rel-P6}}\\
      &= \ve \cdot \vep (1, 2)
      &&\text{by \cref{lemma-B1-holds-in-PT}}\\
      &= (1, 2) \ve \cdot \vep (1, 2)
      && \text{by~\ref{rel-T2}}\\
      &= (1, 2) \cdot \ve \vep \cdot (1, 2)\\
      &= (1, 2) \cdot (1, 2) \vep (1, 2) \vep (1, 2) \cdot (1, 2)
      && \text{by~\ref{rel-P6}}\\
      &= \vep (1, 2) \vep.
    \end{align*}
\end{enumerate}
Thus \cref{theorem-PTn-9-extra} holds by \cref{cor-PT-suffices-2}.

\subsection{The proof of Theorem~\ref{theorem-PTn-main}}

In this section we show that the presentation $\PT$ in \cref{theorem-PTn-main}
defines the partial transformation monoid $PT_n$ when $n \geq 7$.

We do this by showing that the 9 relations in the presentation from
\cref{theorem-PTn-9-extra} hold. These relations
are:~\ref{rel-I3},~\ref{rel-T3},~\ref{rel-T7},~\ref{rel-T8},~\ref{rel-T9},~\ref{rel-P1},~\ref{rel-P5},~\ref{rel-P6},
and~\ref{rel-P7}. The
relations~\ref{rel-T7},~\ref{rel-T8},~\ref{rel-P5}, and~\ref{rel-P6}
belong to $\PT$,
and so we must only verify that the
relations:~\ref{rel-I3},~\ref{rel-T3},~\ref{rel-T9},~\ref{rel-P1},
and~\ref{rel-P7} hold in $\PT$.

Similar to the proof of \cref{theorem-full-transf-main} in
\cref{subsection-proof-full-transf} the key step in the proof of
\cref{theorem-PTn-main} is the following result.

\begin{lemma}\label{lem:ZetaEvenCommute}
  If $\sigma \in A ^*$ represents an (even) permutation in
  $A_{\{3,\dots,n\}}$, then $\sigma^{-1}\zeta\sigma = \zeta$ holds in $\PT$.
\end{lemma}
\begin{proof}
  If $n$ is even, then
  $\alpha\beta^{-1} = (3, 4, n, n - 1, \ldots, 6, 5) =
  \gamma\delta^{-1}$; if $n$ is odd, then
  $\alpha\beta^{-1} = (3, 6, 7, \ldots, n - 1, n)(4, 5) =
  \gamma\delta^{-1}$. So, in either case,
  $\alpha\beta^{-1} = \gamma\delta^{-1}$.
  Throughout the proof we will use the fact that $(1, 2)$ commutes with
  $\alpha$ and $\beta$ without reference. We also
  set $\tau =
  \beta\alpha\beta^{-1}\alpha\beta^{-1}\alpha^{-1}\beta\alpha^{-1}$, as in
  \cref{conj:GenerateAlternatingZeta-odd,conj:GenerateAlternatingZeta-even}.

  From~\ref{rel-P-alpha} and~\ref{rel-P-beta}, it follows that
  \begin{equation}\label{eq-alpha-beta}
    \zeta = (1,2)\beta\alpha^{-1}\zeta\alpha\beta^{-1}
  \end{equation}
  and from~\ref{rel-P-gamma} and~\ref{rel-P-delta} and because
  $\alpha\beta^{-1} =
  \gamma\delta ^{-1}$:
  \begin{equation}\label{eq-gamma-delta}
    \eta = \beta\alpha^{-1}\eta\alpha\beta^{-1}.
  \end{equation}
  Hence
  \begin{align*}
    \zeta
    &=\beta\zeta(1, 2)\eta(1, 2)\beta^{-1}
    &&\text{by~\ref{rel-P-beta}}
    \\
    &=(1, 2)\beta(\beta\alpha^{-1}\zeta \alpha\beta^{-1})(1,
    2)(\beta\alpha^{-1}\eta \alpha\beta^{-1})(1, 2)\beta^{-1}
    &&\text{by~\eqref{eq-alpha-beta} and~\eqref{eq-gamma-delta}}
    \\
    &= (1, 2)\beta^2\alpha^{-1}\zeta(1, 2)\eta(1, 2) \alpha\beta^{-2}.
  \end{align*}
  Thus conjugating $\zeta = (1, 2)\beta^2\alpha^{-1}\zeta(1, 2)\eta(1, 2)
  \alpha\beta^{-2}$ by $\beta^2\alpha^{-1}$ yields
  \begin{align*}
    \alpha\beta^{-2}\zeta\beta^2\alpha^{-1}
    &= (1, 2)\zeta(1, 2)\eta(1, 2)
    \\
    &=(\beta\alpha^{-1}\zeta\alpha\beta^{-1})(1,
    2)(\beta\alpha^{-1}\eta\alpha\beta^{-1})(1, 2)
    &&\text{by~\eqref{eq-alpha-beta} and~\eqref{eq-gamma-delta}}
    \\
    &=\beta\alpha^{-1}\zeta(1, 2)\eta(1, 2)\alpha\beta^{-1}
    \\
    &= \beta\alpha^{-1}(\beta^{-1}\zeta\beta)\alpha\beta^{-1}
    && \text{by~\ref{rel-P-beta}}.
  \end{align*}
  Moreover,
  \begin{align*}
    \alpha\beta^{-2}\zeta\beta^2\alpha^{-1}
    & =   \alpha \beta^{-1}(\zeta (1, 2) \eta (1, 2))\beta\alpha^{-1}
    && \text{by~\ref{rel-P-beta}}
    \\
    &=(1,2)\alpha\beta^{-1}\alpha^{-1}\zeta\alpha\beta\alpha^{-1}
    && \text{by~\ref{rel-P-alpha}}
    \\
    &= \alpha\beta^{-1}\alpha^{-1}(\beta\alpha^{-1}\zeta\alpha\beta^{-1})
    \alpha\beta\alpha^{-1}
    && \text{by~\eqref{eq-alpha-beta}}.
  \end{align*}
  So, by equating both expressions for
  $\alpha\beta^{-2}\zeta\beta^2\alpha^{-1}$,
  we have that
  \begin{equation}\label{eq-dual-tau-zeta}
    \zeta =
    \beta\alpha\beta^{-1}\alpha\beta^{-1}\alpha^{-1}\beta\alpha^{-1}\zeta\alpha\beta^{-1}\alpha\beta\alpha^{-1}\beta\alpha^{-1}\beta^{-1}
    = \tau\zeta\tau^{-1}.
  \end{equation}
  Furthermore, we can see that
  \begin{align*}
    \zeta
    &= (1, 2)\beta\alpha^{-1}\zeta\alpha\beta^{-1}
    && \text{by~\eqref{eq-alpha-beta}}
    \\
    &= (1, 2)\beta\alpha^{-1}\tau\zeta\tau^{-1}\alpha\beta^{-1}
    && \text{by~\eqref{eq-dual-tau-zeta}}
    \\
    &= (1, 2)\beta\alpha^{-1}\tau\alpha\cdot \alpha^{-1}\zeta\alpha
    \cdot\alpha^{-1}\tau^{-1}\alpha\beta^{-1}
    &&\text{since } \alpha\alpha^{-1} = \id
    \\
    &= (1, 2)\beta\alpha^{-1}\tau\alpha\cdot (1, 2)\zeta (1, 2)\eta
    (1, 2)\cdot \alpha^{-1}\tau^{-1}\alpha\beta^{-1}
    &&\text{by~\ref{rel-P-alpha}}
    \\
    &= \beta\alpha^{-1}\tau\alpha\beta^{-1}\zeta\beta
    \alpha^{-1}\tau^{-1}\alpha\beta^{-1}
    &&\text{by~\ref{rel-P-beta}}.
  \end{align*}
  By applying~\eqref{eq-alpha-beta} twice, we obtain that
  \[\zeta =
  \beta\alpha^{-1}\beta\alpha^{-1}\zeta\alpha\beta^{-1}\alpha\beta^{-1}. \]

  Therefore, by the above three equalities for $\zeta$ and
  \cref{conj:GenerateAlternatingZeta-odd,conj:GenerateAlternatingZeta-even}
  it follows that $\zeta
  = \sigma^{-1}\zeta\sigma$ for all $\sigma \in A_{\{3,\dots,n\}}$.
\end{proof}

\begin{lemma}\label{lem:ZetaOddCommute}
  If $\sigma \in A ^*$ represents an odd permutation in
  $S_{\{3,\dots,n\}}$, then $\sigma^{-1}\zeta\sigma = (1,2)\zeta$ holds in
  $\PT$.
\end{lemma}
\begin{proof}
  By \cref{lem:ZetaEvenCommute}, since $\alpha$ is even, $\zeta =
  \alpha^{-1}\zeta\alpha$. Hence~\ref{rel-P-beta}
  and~\ref{rel-P-alpha} imply that
  $\beta^{-1}\zeta\beta = (1,2)\zeta$.
  Since $\beta$ is odd,
  $\sigma = \nu\beta$ for some $\nu \in A_{\{3,\dots,n\}}$. Hence,
  \[ \sigma^{-1}\zeta\sigma = \beta^{-1}\nu^{-1}\zeta\nu\beta =
  \beta^{-1}\zeta\beta = (1,2)\zeta. \qedhere\]
\end{proof}

\begin{lemma}\label{lem:EtaEvenCommute}
  If $\sigma \in A ^*$ represents an (even)
  permutation in $A_{\{2,\dots,n\}}$, then $\sigma^{-1}\eta\sigma =
  \eta$ holds in $\PT$.
\end{lemma}
\begin{proof}
  By~\ref{rel-P-alpha} and~\ref{rel-P-beta}, and because
  $\alpha\beta ^{-1} = \gamma\delta ^{-1}$ (as noted at the start
    of the proof
  of \cref{lem:ZetaEvenCommute}):
  \begin{equation}\label{eq-alpha-beta-2}
    \zeta = (1,2)\delta\gamma^{-1}\zeta\gamma\delta^{-1}.
  \end{equation}
  Also from~\ref{rel-P-gamma} and~\ref{rel-P-delta} we get:
  \begin{equation}\label{eq-gamma-delta-2}
    \eta = \delta\gamma^{-1}\eta\gamma\delta^{-1}
  \end{equation}
  Now, we obtain
  \begin{align*}
    \delta^{-1}\eta\delta &
    =\eta(1, 2)\zeta(1, 2) &&
    \text{by~\ref{rel-P-delta}}\\
    &=(\delta\gamma^{-1}\eta
    \gamma\delta^{-1}) (\delta\gamma^{-1}\zeta \gamma\delta^{-1})(1, 2)
    && \text{by~\eqref{eq-alpha-beta-2} and~\eqref{eq-gamma-delta-2}}\\
    &= \delta\gamma^{-1}(\delta\gamma^{-1}\eta
    \gamma\delta^{-1})(1, 2)(\delta\gamma^{-1}\zeta
    \gamma\delta^{-1})(1, 2)\gamma\delta^{-1}
    && \text{by~\eqref{eq-alpha-beta-2} and~\eqref{eq-gamma-delta-2}}\\
    &= \delta\gamma^{-1}\delta\gamma^{-1}\eta (1, 2)\zeta(1, 2)
    \gamma\delta^{-1} \gamma\delta^{-1}
    &&(1, 2)\text{ and }\delta\gamma^{-1}\text{ commute}
    \\
    &=
    \delta\gamma^{-1}\delta\gamma^{-1}\delta^{-1}\eta\delta\gamma\delta^{-1}
    \gamma\delta^{-1} &&\text{by }\ref{rel-P-delta}
  \end{align*}
  Moreover, by~\ref{rel-P-delta},~\ref{rel-P-gamma},
  and~\eqref{eq-gamma-delta-2}:
  \[
    \delta^{-1}\eta\delta =\gamma^{-1}\eta\gamma =
    \gamma^{-1}(\delta\gamma^{-1}\eta\gamma\delta^{-1})\gamma.
  \]
  So, by equating both expressions for $\delta ^{-1}\eta\delta$ we have that
  \begin{equation}\label{eq-tau-eta}
    \eta =
    \delta\gamma\delta^{-1}\gamma\delta^{-1}\gamma^{-1}\delta\gamma^{-1}\eta\gamma\delta^{-1}\gamma\delta\gamma^{-1}\delta\gamma^{-1}\delta^{-1}
    = \rho\eta\rho^{-1}
  \end{equation}
  where $\rho =
  \delta\gamma\delta^{-1}\gamma\delta^{-1}\gamma^{-1}\delta\gamma^{-1}$
  as in
  \cref{conj:GenerateAlternatingEta-odd,conj:GenerateAlternatingEta-even}.
  Furthermore, we can see that
  \begin{align*}
    \eta &= \delta\gamma^{-1}\eta\gamma\delta^{-1}
    &&\text{by~\eqref{eq-gamma-delta-2}} \\
    &= \delta\gamma^{-1}\rho\eta\rho^{-1}\gamma\delta^{-1}
    &&\text{by~\eqref{eq-tau-eta}} \\
    &=
    \delta\gamma^{-1}\rho\gamma(\gamma^{-1}\eta\gamma)\gamma^{-1}\rho^{-1}\gamma\delta^{-1}
    && \gamma\gamma^{-1} = \id \\
    &=\delta\gamma^{-1}\rho\gamma\delta^{-1}\eta\delta
    \gamma^{-1}\rho^{-1}\gamma\delta^{-1}
    &&\text{by~\ref{rel-P-gamma} and~\ref{rel-P-delta}}.
  \end{align*}
  Moreover, by applying~\eqref{eq-gamma-delta-2} twice, we obtain that
  \[\eta =
    \delta\gamma^{-1}\delta\gamma^{-1}\eta\gamma\delta^{-1}\gamma\delta^{-1}.
  \]
  Therefore, by the above three equalities for $\eta$ and
  \cref{conj:GenerateAlternatingEta-odd,conj:GenerateAlternatingEta-even}
  it follows that $\eta = \sigma^{-1}\eta\sigma$ for $\sigma \in
  A_{\{2,\dots,n\}}$.
\end{proof}

In the next lemma, we show that $\eta$ is fixed by conjugation by any odd
permutation.

\begin{lemma}\label{lem:EtaOddCommute}
  If $\sigma \in A^*$ represents an odd permutation in $S_{\{2,\dots,n\}}$,
  then $\sigma^{-1}\eta\sigma = \eta$ holds in $\PT$.
\end{lemma}
\begin{proof}
  By \cref{lem:EtaEvenCommute}, $\eta = \delta^{-1}\eta\delta$
  since $\delta$
  is even. It follows that $\gamma^{-1}\eta\gamma = \eta$
  by~\ref{rel-P-gamma} and~\ref{rel-P-delta}. If $\sigma\in A^*$ represents
  any odd permutation, then there exists $\nu \in
  A_{\{2,\dots,n\}}$ such that
  $\sigma = \nu\gamma$. Hence
  \[
    \sigma^{-1}\eta\sigma = \gamma^{-1}\nu^{-1}\eta\nu\gamma =
    \gamma^{-1}\eta\gamma = \eta. \qedhere
  \]
\end{proof}

\begin{lemma}\label{lem:etaIdempotent}
  The relation $\eta^2 = \eta$ holds in $\PT$.
\end{lemma}
\begin{proof}
  By applying \cref{lem:EtaEvenCommute} to (\ref{rel-P-delta}), we
  obtain that $\eta = \eta(1,2)\zeta(1,2)$ because $\delta$ is even.
  Similarly, by applying \cref{lem:ZetaEvenCommute}
  to~\ref{rel-P-alpha}, we obtain that $\zeta =
  (1,2)\zeta(1,2)\eta(1,2)$ since
  $\alpha$ is even. Conjugating $\eta = \eta(1, 2)\zeta(1, 2)$
  by any odd $\sigma \in S_{\{3,\dots,n\}}$, yields
  \begin{align*}
    \sigma^{-1}\eta\sigma & =
    \sigma^{-1}\eta(1, 2)\zeta(1, 2)\sigma \\
    &= \sigma^{-1}\eta(1, 2)\sigma\sigma^{-1}\zeta(1, 2)\sigma \\
    &=\sigma^{-1}\eta\sigma
    (1, 2)\sigma^{-1}\zeta\sigma(1, 2) && \sigma \text{ and }(1, 2)
    \text{ commute}\\
    &= \eta(1, 2)(1, 2)\zeta(1, 2) &&\text{by
    \cref{lem:ZetaOddCommute,lem:EtaOddCommute}}\\
    &= \eta(1, 2)\zeta(1, 2)\eta(1, 2)(1, 2) &&\zeta =
    (1, 2)\zeta(1, 2)\eta(1, 2)\\
    &= \eta\eta && \eta(1, 2)\zeta(1, 2) = \eta.
  \end{align*}
  Hence, $\eta^2 = \eta$ holds in $\PT$.
\end{proof}

We can now show that the
relations~\ref{rel-P1},~\ref{rel-I3},~\ref{rel-T3},~\ref{rel-T9},
and~\ref{rel-P7} hold in $\PT$, which
completes the proof of \cref{theorem-PTn-main}:

\begin{enumerate}
  \item [\eqref{rel-P1}]
    By \cref{lem:ZetaOddCommute} applied to~\ref{rel-P-beta},
    since $\beta$ is odd,
    $(1, 2)\zeta = \zeta(1, 2)\eta(1, 2)$. Then,
    \begin{align*}
      \zeta &= (1, 2)\zeta\cdot (1, 2)\eta(1, 2) \\
      &= \zeta(1, 2)\eta(1, 2)\cdot (1, 2)\eta(1, 2) && (1, 2)\zeta =
      \zeta(1, 2)\eta(1, 2)\\
      &= \zeta(1, 2)\eta^2(1, 2) \\
      &= \zeta(1, 2)\eta(1, 2) &&\text{by \cref{lem:etaIdempotent}}.
    \end{align*}
  \item [\eqref{rel-I3}] That $(2, \ldots, n)\eta = \eta (2, \ldots, n)$
    follows immediately from \cref{lem:EtaEvenCommute} when $n$ is even and
    \cref{lem:EtaOddCommute} when $n$ is odd.

  \item [\eqref{rel-T3}] The relation $(3, 4)\zeta= \zeta(3, 4)$
    holds because
    \begin{align*}
      (3, 4)\zeta (3, 4) & = (1, 2)\zeta && \text{by
        \cref{lem:ZetaOddCommute}
      since $(3, 4)$ is odd} \\
      & = \beta^{-1}\zeta \beta && \text{by
      \cref{lem:ZetaOddCommute} since $\beta$ is odd}\\
      & = \zeta(1,2)\eta(1,2) && \text{by~\ref{rel-P-beta}} \\
      & = \zeta && \text{by~\ref{rel-P1}}.
    \end{align*}

  \item [\eqref{rel-T9}] We must show that $(3, 4, \ldots, n)(1, 2)\zeta=
    \zeta(3, 4, \ldots, n)$ holds in $\PT$. If $n$ is even, then $(3, 4,
    \ldots, n)$ conjugates $\zeta$ to $(1, 2)\zeta$ by
    \cref{lem:ZetaOddCommute}. Hence in this case the relation holds.

    If $n$ is odd, then $(3, 4, \ldots, n)$ conjugates $\zeta$ to $\zeta$ by
    \cref{lem:ZetaEvenCommute}. But $\zeta = (1, 2)\zeta$ as shown
    in the proof that~\ref{rel-T3} holds in $\PT$ above.

  \item [\eqref{rel-P7}] That $(2, 3)\eta(2, 3) = \eta
    (1, 2)\zeta(1, 2)$ holds
    follows from:
    \begin{align*}
      (2, 3)\eta (2, 3) & = \eta && \text{by \cref{lem:EtaOddCommute}
      since $(2, 3)$ is odd} \\
      & = \delta^{-1}\eta \delta && \text{by
      \cref{lem:EtaEvenCommute} since $\delta$ is even}\\
      & = \eta(1,2)\zeta(1,2) && \text{by~\ref{rel-P-delta}}.
    \end{align*}
\end{enumerate}

\subsection{The $n = 1$ to $6$ cases}\label{section-PTn-leq-4}

As in \cref{section-In-leq-3,section-Tn-leq-4}, we consider
presentations for
$PT_n$ where $n = 1, \ldots, 6$. Note that $PT_1 \cong I_1$, and so the
presentation for $I_1$ given in \cref{section-In-leq-3} also defines $PT_1$.
\GAP and Python code for performing the computations
mentioned in this section can be found in~\cite{libsemigroupsShortCaseStudy}.

For $n = 2$, it can be verified computationally that the presentation
\[
  \presn{a_2, \ve, \vep}{a_2^2 = 1,\quad  a_2\ve = \ve,\quad  \ve a_2 \vep a_2
  = \ve,\quad  \vep a_2 \ve a_2 = \vep,\quad a_2 \vep a_2 \vep a_2 = \ve \vep}
\]
defines $PT_2$. Using arguments similar to those in
\cref{section-In-leq-3,subsection-PTn-lower-bound}, it can be shown
that any presentation
requires at least three relations of rank $n - 1$ and one
relation of rank $n -
2$. It follows that the minimum number of non-$S_n$
relations in any
presentation for $PT_2$ is $4$.

For $n = 3$, it can be verified that a presentation for $S_3$, together with
the seven non-$S_n$
relations~\ref{rel-I2},~\ref{rel-T2},~\ref{rel-T8},~\ref{rel-P1},~\ref{rel-P2},~\ref{rel-P5},
and~\ref{rel-P6},
defines $PT_3$. Using
arguments similar to those in \cref{subsection-PTn-lower-bound}
and for $I_3$ in
\cref{section-In-leq-3}, it can be shown that any presentation for $PT_3$
requires at least seven non-$S_n$ relations. The minimum number of
non-$S_n$ relations in any presentation for $PT_3$ is
therefore seven.

For $n = 4$, $5$, and $6$ it can be verified that the following presentation
with eight non-$S_n$ relations defines $PT_n$ where $\presn{A}{R}$ is
Carmichael's~\cite{Carmichael1937aa} monoid presentation
in~\eqref{eq-carmichael}
for the symmetric group. It therefore follows that any monoid presentation
$\presn{A}{R}$ can be used in place of those from~\eqref{eq-carmichael}. By
\cref{thm-atleast8-for-pt}, it follows that this is the minimum
possible number of such relations.

Suppose that $n = 4$, $5$, or $6$ and that $\presn{A}{R}$ is any monoid
presentation for the symmetric group $S_n$ of degree $n$. Then the
presentation
with generators $A\cup \{\ve, \vep\}$, and relations $R$ together with the
irredundant relations:
\setlength{\columnsep}{2cm}
\begin{multicols}{2}
  \begin{enumerate}[label=\rm (Y\arabic*), ref=\rm Y\arabic*]
    \item [\rm (\ref{rel-I3})]
      $\vep  (2,  3, \ldots, n)  = (2,  3, \ldots, n)   \vep$;
    \item [\rm (\ref{rel-T7})]
      $\ve (1, n)(2, 3)  \ve  (1, n)(2, 3)  = (1, n)(2, 3) \ve
      (1, n)(2, 3)  \ve$;
    \item [\rm (\ref{rel-T8})]
      $(2, 3)  \ve  (2, 3)  \ve(2, 3)  \ve(2, 3) = \ve  (2, 3)   \ve$;
    \item [\rm (\ref{rel-T9})]
      $(3,  4, \ldots,  n) (1,  2)
      \ve = \ve  (3,  4,  \ldots, n)$;
    \item [\rm (\ref{rel-P5})]
      $\ve  (1,  3)   \vep  (1,  3)  = (1,  3)   \vep  (1,  3)   \ve$;
    \item[\rm (\ref{rel-P6})]
      $\ve \vep = (1,  2)   \vep  (1,  2)   \vep  (1,  2) $;
    \item \label{rel-Y1}
      $(2,  3)   \vep  (2,  3)  = \vep  (1,  2)   \ve  (1,  2) $;
    \item \label{rel-Y2}
      $(3, 4) \ve (3, 4) = \ve  (1,  2)  \vep  (1,  2)$.
  \end{enumerate}
\end{multicols}

\section*{Glossary of relations}
For the readers convenience, we list all the named relations below.
\setlength{\columnsep}{2cm}
\begin{multicols}{2}
  \begin{enumerate}
    \item[(\ref{rel-I1})]  $\vep^2 = \vep$
    \item[(\ref{rel-I2})] $(2, 3) \vep = \vep (2, 3)$
    \item[(\ref{rel-I3})]  $(2, 3, \ldots, n) \vep = \vep(2, 3, \ldots, n)$
    \item[(\ref{rel-I4})]  $\vep  (1, 2)   \vep  (1, 2)  =  \vep(1, 2)\vep$
    \item[(\ref{rel-I5})]  $(1, 2)   \vep  (1, 2)   \vep =  \vep(1, 2)\vep$
    \item[(\ref{rel-I6})]   $(2, 3, \ldots, n)\vep = \vep^2 (2, 3, \ldots, n)$
    \item[(\ref{rel-I7})]  $(1, 2) \vep(1, 2)\vep(1, 2)\vep (1, 2) =
      \vep (1, 2) \vep$
    \item[(\ref{rel-T1})]   $\ve  (1, 3) \ve  (1, 3) = \ve$
    \item[(\ref{rel-T2})]   $(1, 2)  \ve = \ve$
    \item[(\ref{rel-T3})]   $(3, 4) \ve = \ve  (3, 4)$
    \item[(\ref{rel-T4})]   $(3, 4, \ldots, n)  \ve = \ve  (3, 4, \ldots, n)$
    \item[(\ref{rel-T6})]   $\ve (2, 3) \ve (2, 3) = \ve (2, 3) \ve$
    \item[(\ref{rel-T5})]   $(2, 3) \ve (2, 3) \ve = \ve (2, 3) \ve$
    \item[(\ref{rel-T7})]   $(1, n)(2, 3) \ve  (1, n)(2, 3) \ve
      = \ve (1, n)(2, 3)  \ve  (1, n)(2, 3)$
    \item[(\ref{rel-T8})]  $(2, 3)  \ve  (2, 3)  \ve(2, 3)  \ve(2, 3)
      = \ve  (2, 3)   \ve$
    \item[(\ref{rel-T9})]  $(3, 4, \ldots, n) (1, 2) \ve = \ve  (3,
      4, \ldots, n)$
    \item[(\ref{rel-T-alpha})]  $\alpha ^{-1} \ve  \alpha = \ve  (1,
      3)  \ve  (1, 3)$
    \item[(\ref{rel-T-beta})]  $\beta^{-1} \ve \beta =(1, 2)  \ve (1,
      3)   \ve (1, 3)$
    \item[(\ref{rel-P1})]  $\ve (1, 2) \vep (1, 2)  = \ve$
    \item[(\ref{rel-P2})]  $\vep (1, 2)  \ve (1, 2)  = \vep$
    \item[(\ref{rel-P3})]  $\ve \vep = \vep (1, 2)  \vep$
    \item[(\ref{rel-P4})]  $(1, 3)  \ve (1, 2, 3)  \ve = \ve (1, 2, 3)  \ve$
    \item[(\ref{rel-P5})]  $(1, 3)  \vep (1, 3)  \ve = \ve (1, 3)  \vep (1, 3)$
    \item[(\ref{rel-P6})]  $(1,  2)   \vep  (1,  2)   \vep  (1,  2) = \ve \vep$
    \item[(\ref{rel-P7})]  $(2,  3)   \vep  (2,  3)  = \vep  (1,  2)
      \ve  (1,  2) $
    \item[(\ref{rel-P-alpha})]  $\alpha ^{-1}\ve \alpha = (1, 2)\ve
      (1, 2)\vep (1, 2)$
    \item[(\ref{rel-P-beta})]  $\beta ^{-1}\ve \beta = \ve (1, 2)\vep (1, 2)$
    \item[(\ref{rel-P-gamma})]  $\gamma^{-1}\vep \gamma = \vep (1, 2)\ve (1, 2)$
    \item[(\ref{rel-P-delta})]  $\delta^{-1}\vep \delta = \vep (1, 2)\ve (1, 2)$
    \item[(\ref{rel-Y1})]  $(2,  3)   \vep  (2,  3)  = \vep  (1,  2)
      \ve  (1,  2) $
    \item[(\ref{rel-Y2})]  $(3, 4) \ve (3, 4) = \ve  (1,  2)  \vep  (1,  2)$
  \end{enumerate}
\end{multicols}

\section*{Acknowledgments}
The authors thank the anonymous referee for their suggested improvements.
The third author acknowledges St Leonard's College, University of St
Andrews for his PhD funding.

\printbibliography{}

\end{document}